\documentclass[a4paper,11pt, twoside]{article}
\usepackage{mathrsfs}
\usepackage{bbm}
\usepackage{cite}
\usepackage{amsfonts}
\usepackage{amsmath,amssymb}
\usepackage{multicol}

\usepackage{latexsym}
\usepackage{graphicx}
\usepackage{dcolumn}
\usepackage{bm}
\usepackage{fancyhdr}
\usepackage{mathrsfs}
\usepackage{wasysym}
\usepackage{epsf}
\usepackage{epsfig}
\usepackage{psfrag}
\usepackage{ifpdf}
\usepackage[
  bookmarks=true,
  bookmarksopen=true,
  breaklinks=true,
  colorlinks=true,
  linkcolor=blue,anchorcolor=blue,
  citecolor=blue,filecolor=blue,
  menucolor=blue,
  urlcolor=blue]{hyperref}

 \numberwithin{equation}{section}
\newtheorem{theorem}{Theorem}[section]
\newtheorem{lemma}[theorem]{Lemma}

\newenvironment{proof}[1][Proof]{\noindent\textbf{#1.} }{\hfill $\Box$}

\allowdisplaybreaks

\allowdisplaybreaks \numberwithin{equation}{section}
\makeatletter
\begin{document}

\author{\small Yuan Li
	\footnote{{E-mail addresses}: yli2021@ccnu.edu.cn (Y. Li)}
	\setcounter{footnote}{0}
	\\\small  School of Mathematics and Statistics, Central China Normal University, Wuhan 430079,  China}
\date{}
\title{\textbf{Blowup dynamics for inhomogeneous mass critical half-wave  equation }}
\date{  }
\maketitle

\begin{abstract}
We consider the focusing inhomogeneous  mass critical half-wave equation  in one dimension. Under the mild conditions of the inhomogeneous factor, we  show that  the existence of the radial blowup solutions with ground state mass $\|u_0\|_2=\|Q\|_2$, where $Q$ is the unique positive ground state solution of equation $DQ+Q=Q^3$ and obtain the blowup rate $\|D^{\frac{1}{2}}u(t)\|_{L^2}\sim\frac{1}{|t|}$ as $t\nearrow 0^-$.


\end{abstract}

\section{Introduction and the main result}\label{section:1}
\subsection{Introduction}
We deal in this paper with a one-dimensional focusing mass critical half-wave equation with an inhomogeneous nonlineraity
\begin{equation}\label{equ-1-hf-in}
\begin{cases}
	i\partial_tu=Du-k(x)|u|^2u,\,\,(t,x)\in\mathbb{R}\times\mathbb{R},\\
	u(t_0,x)=u_0(x),\ u_0:\mathbb{R}\rightarrow\mathbb{C}.
\end{cases}
\end{equation}
Here $\widehat{(Df)}(\xi)=|\xi|\hat{f}(\xi)$ denotes the first-order nonlocal fractional derivative and for some smooth bounded inhomogeneity $k:\mathbb{R}\rightarrow\mathbb{R}_+^*$ and some real number $t_0<0$. This kind of problem arises naturally in
turbulence phenomena, wave propagation, continuum limits of lattice systems, and models for gravitational collapse in astrophysics \cite{majda1997,majda2001,Weinstein1987,Ionescu2014,Frank-lenzmann2013,eckhaus1983,klein2014,k-lenzmann2013}. From mathematical point of view, it is a canonical model to break the group of symmetries of the $k\equiv1$ homogeneous case.
	

Let us start with recalling some well-known facts in the homogeneous case $k\equiv1$:
\begin{equation}\label{equ-1-hf}
	\begin{cases}
		i\partial_tu=Du-|u|^2u,\,\,(t,x)\in\mathbb{R}\times\mathbb{R},\\
		u(t_0,x)=u_0(x),\ u_0:\mathbb{R}\rightarrow\mathbb{C},
	\end{cases}
\end{equation}
where $t_0<0$. All $H^{1/2}$ solutions must  satisfy the conservation laws of  $L^2$-norm, energy and momentum
\begin{align*}
	\text{$L^2$-norm}\ \ M(u)&=\int_{\mathbb{R}}|u(t,x)|^2dx=M(u_0);	\\
	\text{Energy}\ \ E(\psi)&=\frac{1}{2}\int_{\mathbb{R}}\bar{u}(t,x)Du(t,x)dx-\frac{1}{4}\int_{\mathbb{R}}|u(t,x)|^{4}dx=E(u_0);\\
	\text{Momentum}\ \
	P(u)&=\Im \left(\int\nabla u(t,x)\bar{u}(t,x)dx\right)=P(u_0),
\end{align*}
and a group of $H^{1/2}$ symmetries leaves the flow invariant: if $u(t,x)$ solves \eqref{equ-1-hf}, then for any $(\lambda_0,\tau_0,x_0,\gamma_0)\in\mathbb{R}_+^*\times\mathbb{R}\times\mathbb{R}\times\mathbb{R}$, so does
\begin{align}\label{symmetry-in}
	v(t,x)=\lambda_0^{\frac{1}{2}}u(\lambda_0t+\tau_0,\lambda_0 x+x_0)e^{i\gamma_0}.
\end{align}
Following from  \cite{KLR2013}, let $Q$ be the unique  positive radial ground state solution to
\begin{equation}\label{equ-elliptic-in}
	DQ+Q-Q^3=0.
\end{equation}
Then the variational characterization of $Q$ ensures  that initial data $u_0\in H^{1/2}(\mathbb{R})$ with $\|u_0\|_{L^2}<\|Q\|_{L^2}$ yield global and bounded solutions $T=+\infty$, see \cite{KLR2013}. On the other hand, finite time blowup may occur for data $\|u_0\|_{L^2}\geq\|Q\|_{L^2}$.
At the critical mass threshold, the symmetries \eqref{symmetry-in} yields a minimal mass blow-up solution (see \cite{KLR2013})
\begin{align*}
	u(t,x)\sim\frac{1}{|t|}Q\left(\frac{x-\alpha}{t^2}\right)e^{i\gamma(t)},\,\,\|u(t)\|_{L^2}=\|Q\|_{L^2},
\end{align*}
which blows up at $t=0$ and the blow up speed is
\begin{align*}
	\|D^{\frac{1}{2}}u(t)\|_{L^2}\sim \frac{C(u_0)}{|t|}\,\,\text{as}\,\,t\rightarrow 0^{-}.
\end{align*}
Recently, Georgiev and Li (the author of this paper) \cite{Georgiev-Li-2022-CPDE,Georgiev-Li-2021-JFA} also obtained the ground state blowup solution to the mass critical half-wave equation in two and three dimensional cases. But unlike the mass critical NLS \cite{Merle-1993-Duke,Raphael2011-Jams}, the uniqueness for this minimal (ground state) mass blow-up
solution is still not known.	

The structure of the problem is similar to the mass critical nonlinear Schr\"odinger equation
\begin{eqnarray}\label{equ-classical}
iu_t+\Delta u+|u|^{\frac{4}{N}}u=0.
\end{eqnarray}
At the mass critical threshold, the pseudo-conformal symmetry yields a minimal mass blowup solution
\begin{align*}
    S(t,x)=\frac{1}{|t|^{N/2}}Q\left(\frac{x}{t}\right)e^{i\frac{|x|^2}{4t}-\frac{i}{t}},\,\,\,\|S(t)\|_{L^2}=\|Q\|_{L^2},
\end{align*}
where $Q$ is the unique positive ground state of problem \eqref{equ-classical}, which blows up at $t=0$. Merle \cite{Merle-1993-Duke} proved the uniqueness of the critical mass blowup solution $u\in H^1$ with $\|u_0\|_{L^2}=\|Q\|_{L^2}$ and blowing up at $t=0$ is equal to $S(t)$ up to the symmetries of the flow. A robust dynamical approach for the proof of both existence and uniqueness has been developed by Rapha\"el and Szeftel \cite{Raphael2011-Jams} for the inhomogeneous problem
\begin{eqnarray}\label{equ-classical-non}
iu_t+\Delta u+k(x)|u|^{\frac{4}{N}}u=0.
\end{eqnarray}
in two dimensions. Also,  Banica,  Carles and  Duyckaerts \cite{Banica-2011-CPDE} obtained the existence of the minimal blowup solution of problem \eqref{equ-classical-non} in the one and two dimensional cases.

Inspired by \cite{Banica-2011-CPDE,Raphael2011-Jams,KLR2013}, in this paper we study the inhomogeneous $L^2$-critical half-wave equation \eqref{equ-1-hf-in}. By a similar argument as the homogeneous half wave equation \cite{KLR2013} and the cubic Szeg\H{o} equation \cite{Gerard-2010}, the Cauchy problem \eqref{equ-1-hf-in} is still locally wellposed in the Sobolev space $H^{s}(\mathbb{R})$, $s>\frac{1}{2}$ and local existence in the energy space $H^{1/2}(\mathbb{R})$, that is, there exists a unique solution $u\in C^0([t_0,T),H^s(\mathbb{R}))$, where $s\geq \frac{1}{2}$. with its maximal time of existence $t_0<T\leq+\infty$, and
\begin{align}
    T<+\infty\,\,\text{implies}\,\,\lim_{t\to T^-}\|u(t)\|_{H^{1/2}}=+\infty.
\end{align}
The conservation of the momentum no longer holds but the $L^2$-norm and energy are still conserved:
\begin{align*}
	&\text{$L^2$-norm:}\,\,\int_{\mathbb{R}}|u(t,x)|^2=\int_{\mathbb{R}}|u_0(x)|^2,\\
	&\text{Energy:}\,\,E(u(t,x))=\frac{1}{2}\int_{\mathbb{R}}\bar{u}(t,x)Du(t,x)dx
	-\frac{1}{4}\int_{\mathbb{R}}k(x)|u(t,x)|^{4}dx=E(u_0).
\end{align*}
For the higher dimensional case, Bellazzini,  Georgiev and Visciglia \cite{BGV2018} obtained the  local well-posedness in  $H^1_{rad}(\mathbb{R}^N)$, $N\geq2$; Hidano and Wang \cite{Hidano-2019-sel} studied the local well-posedness in  $H^{\frac{1}{2}+\epsilon}_{rad}(\mathbb{R}^N)$, $N\geq2$ and $\epsilon>0$ is small.
The  canonical effect of the inhomogenity is to completely destory the group of symmetry \eqref{symmetry-in} and in the sense \eqref{equ-1-hf-in} is a model to analyze the properties of half-wave equation in the absence of symmetries.

From the standard variational techniques, we can obtain the criterion of global existence for \eqref{equ-1-hf-in}: given $\kappa>0$, let $Q_{\kappa}$ be
\begin{align}\label{intro-1-in}
	Q_{\kappa}(x)=\frac{1}{\kappa^{1/2}}Q(x),
\end{align}
and let
\begin{align*}
	\kappa=\max_{x\in\mathbb{R}}k(x)<\infty.
\end{align*}
Then the initial data with
\begin{align}\label{intro-2-in}
	\|u_0\|_{L^2}<M_k=\|Q_{\kappa}\|_{L^2}
\end{align}
yield global and $H^{1/2}$ bounded solutions while finite time blow-up may occur for $\|u_0\|_{L^2}\geq M_k$.
	
\subsection{Statement of the result}

Let us now fix our assumptions on $k$ where we normalize without loss of generality the supremum $k_2=1$:
\begin{align}\label{assumption-1}
	0<k_1\leq k(x)\leq1\,\,\text{and}\,\,\max_{x\in\mathbb{R}}k(x)=1\,\,\text{is attained}.
\end{align}
From \eqref{intro-1-in}, \eqref{intro-2-in}, the critical mass is then $$M_k=\|Q\|_{L^2}.$$
The following theorem is the main result of this paper.
\begin{theorem}(Existence of Ground state mass blowup elements)\label{Theorem-1-hf-2}
Assume that the function $k\in C^2(\mathbb{R})$ is even and satisfies \eqref{assumption-1}. For  $E_0\in \mathbb{R}_+^*$, there exist $t^*<0$ independent of $E_0$ and a ground state mass solution $u\in C^0([t^*,0);H^{1/2}(\mathbb{R}))$ of equation \eqref{equ-1-hf-in} with
\begin{align}\notag
\|u\|_2=\|Q\|_2,\ E(u)=E_0,
\end{align}
which blows up at time $T=0$. More precisely, it holds that
\begin{align}
u(t,x)-\frac{1}{\lambda(t)^{1/2}}Q\left(\frac{x}{\lambda(t)}\right)
	e^{i\gamma(t)}\rightarrow0\ \text{in}\ L^2(\mathbb{R})\ \text{as}\ t\rightarrow0^-,
\end{align}
where
\begin{align}\notag
\lambda(t)=\lambda^*t^2+\mathcal{O}(t^5),\,\, \gamma(t)=\frac{1}{\lambda^*|t|}+\mathcal{O}(t),
\end{align}
with some constant $\lambda^*>0$, and the blowup speed is given by:
\begin{align}
\|D^{\frac{1}{2}}u(t)\|_2\sim\frac{C(u_0)}{|t|}\ \text{as}\ t\rightarrow0^{-},
\end{align}
where $C(u_0)>0$ is a constant  depending only  on the initial data $u_0$.
\end{theorem}

Comments on this result.

1. For the inhomogeneous case, the function $k$ destroy the symmetry of the half-wave equation. Due to the lack of  homogeneity of the nonlinear terms, we cannot construct the blowup profile like the inhomogeneous Schr\"odinger equation \cite{Raphael2011-Jams} or the homogeneous half wave equation \cite{Georgiev-Li-2021-JFA,Georgiev-Li-2022-CPDE,KLR2013}. Fortunately, we can use another method to construct the blowup profile in the general case.  But  the translation parameter is not well controlled since the inhomogeneous half-wave equation does not satisfy the momentum law.  On the other hand, the pseudoconformal symmetry plays an important role in  controlling the translation parameter in the inhomogeneous NLS equation, see \cite{Raphael2011-Jams}.  Hence, in the present paper, we only consider the radial case.

2. The degenerate case. If the inhomogeneous factor $k$ satisfies $$k^\prime(x_0)=k^{\prime\prime}(x_0)=0,$$
the method presented in this paper should also be able to treat this degenerate case, which should in fact be easier to handle.

3. In the present paper, the inhomogeneous factor is not sharp, unlike the inhomogeneous NLS equation \cite{M1996}, it seems difficult to obtain the nonexistence result. On the other hand,  since the general criterion for blowup solutions for $L^2$-critical and $L^2$-supercritical half-wave equation is still an open problem (see \cite{lenzmann-2016blowup} for more details).

This paper is organized as follows: in Section 2,  we construct the high order approximate  solution $Q_{\mathcal{P}}$ of the renormalized equation; in Section 3, we decompose the solution and estimate the modulation parameters; in Section 4, we establish a refined energy/virial type estimate; in Section 5, we apply the energy estimate to establish a bootstrap argument that will be needed in the construction of ground state mass blowup solutions; in Section 6, we prove the Theorem \ref{Theorem-1-hf-2}; The Section 7 is Appendix.\\
\textbf{Notations}\\
- $(f,g)=\int \bar{f}g$ as the inner product on $L^2(\mathbb{R})$.\\
- $\|\cdot\|_{L^p}$ denotes the $L^p(\mathbb{R})$ norm for $p\geq 1$.\\
- $\widehat{f}$ denotes the Fourier transform of function $f$.\\
- We shall use $X\lesssim Y$ to denote that $X\leq CY$ holds, where the constant $C>0$ may change from line to line, but $C$ is allowed to depend on universally fixed quantities only.\\
- Likewise, we use $X\sim Y$ to denote that both $X\lesssim Y$ and $Y\lesssim X$ hold.

\section{Approximate Blowup Profile}
This section is devoted to the construction of the approximate blowup profile. For a sufficiently regular function $f:\mathbb{R}\rightarrow\mathbb{C}$, we define the generator of $L^2$ scaling given by
\begin{align}\notag
	\Lambda f:=\frac{1}{2}f+x\cdot\nabla f.
\end{align}
Note that the operator $\Lambda$ is skew-adjoint on $L^2(\mathbb{R})$, that is, we have
\begin{align}\notag
	(\Lambda f, g)=-(f,\Lambda g).
\end{align}
We write $\Lambda^kf$, with $k\in\mathbb{N}$, for the iterates of $\Lambda$ with the convention that $\Lambda^0f\equiv f$.

From now on and for the rest of this paper, we assume that $k$ satisfies {\bf Assumption} \eqref{assumption-1}. Moreover, without loss of generality, we assume that $k$ attains its maximum at $x_0=0$ which is nondegenerate
\begin{align}\label{asumption-2}
    k(0)=1,\,\,k^\prime(0)=0,\,\,\,k^{\prime\prime}(0)<0.
\end{align}
	
In some parts of this paper, it will be convenient to identity any complex-valued function $f:\mathbb{R}\rightarrow\mathbb{C}$, so that, the complex-valued function $f$ can be write as follows:
\begin{align}\notag
	f=f_1+if_2.
\end{align}

We start with a general observation: If $u=u(t,x)$ solves \eqref{equ-1-hf-in}, then we define the function $v=v(s,y)$ by setting
\begin{align}
u(t,x)=\frac{1}{\lambda^{\frac{1}{2}}(t)}v\left(s,\frac{x}{\lambda(t)}\right)e^{i\gamma(t)},\ \ \frac{ds}{dt}=\frac{1}{\lambda(t)}.
\end{align}
It is easy to check that $v=v(s,y)$ with $y=\frac{x}{\lambda}$ satisfies
\begin{equation}\label{equ-transform-hf-N}
	i\partial_sv-Dv-v+k(\lambda(t)y)|v|^2v=i\frac{\lambda_s}{\lambda}\Lambda v++\tilde{\gamma_s}v,
\end{equation}
where we set $\tilde{\gamma_s}=\gamma_s-1$. Here the operator $D$ is understood as $D=D_y$. Following the slow modulated ansatz strategy developed in \cite{Raphael-2009-cpam,Raphael2011-Jams,KLR2013}, we freeze the modulation
\begin{align}\label{mod1}
		-\frac{\lambda_s}{\lambda}=b.
\end{align}
And we look for an approximate solution of the form
\begin{align}\label{eq.der1}
	v(s,y) = Q_{\mathcal{P}(s)}(y),\ \mathcal{P}(s)=(b(s),\lambda(s)),
\end{align}
with an expansion
\begin{align*}
    Q_{\mathcal{P}(s)}(y)=Q(y)+\sum_{k+l\geq1}b^k\lambda^l\mathbf{R}_{k,l},
\end{align*}
where
\begin{align*}
    \mathbf{R}_{k,l}=T_{k,l}+iS_{k,l}.
\end{align*}
These asymptotic expansions suggests to define $b(s)$ so that
\begin{align}\label{as11}
	b_s=-\frac{b^2}{2}\,\,\,\,\text{and}\,\,\,\,\lambda_s=-\lambda b.
\end{align}	
Therefore, our purpose is  to construct a high order approximation $Q(y,b,\lambda)=Q_{\mathcal{P}}$ that is a solution to
\begin{align}\notag
-i\frac{b^2}{2}\partial_bQ_{\mathcal{P}}-ib\lambda\partial_{\lambda}Q_{\mathcal{P}}-DQ_{\mathcal{P}}-Q_{\mathcal{P}}+ib\Lambda Q_{\mathcal{P}}+
k(\lambda(t)y)|Q_{\mathcal{P}}|^2Q_{\mathcal{P}}=0,
\end{align}
where $\mathcal{P}=(b,\lambda)$ and $|\mathcal{P}|$ is close to $0$.
	
We have the following result about an approximate blowup profile $Q_{\mathcal{P}}$, parameterized by $\mathcal{P}=(b,\lambda)$, around the ground state $Q$ of the homogeneous half-wave equation.
\begin{lemma}(\textbf{Approximate Blowup Profile})\label{lemma-3app}
There exists a small constant $\eta^*>0$ such that for all
$|\mathcal{P}|=|(b,\lambda)|\leq \eta^*$. There exists a smooth function $Q_{\mathcal{P}}=Q_{\mathcal{P}}(x)$ of the form
\begin{align}\label{approximate-solution-in}
	Q_{\mathcal{P}}=&Q+bS_{1,0}+b^2T_{2,0}+\lambda^2T_{0,2}+b^3S_{3,0}+b^4T_{4,0}
\end{align}
that satisfies the equation
\begin{align}\label{equ-3approxiamte}
&-i\frac{b^2}{2}\partial_bQ_{\mathcal{P}}-ib\lambda\partial_{\lambda}Q_{\mathcal{P}}-DQ_{\mathcal{P}}-Q_{\mathcal{P}}+ib\Lambda Q_{\mathcal{P}}+k(\lambda(t)y)|Q_{\mathcal{P}}|^2Q_{\mathcal{P}}=-\Phi_{\mathcal{P}},
\end{align}
Here the functions $\{\mathbf{R}_{k,l}\}_{0\leq k\leq4,0\leq l\leq 1}$ satisfy the following regularity and decay bounds:
\begin{align}\label{3app-regulairty}
\|\mathbf{R}_{k,l}\|_{H^m}+\|\Lambda\mathbf{R}_{k,l}\|_{H^m}+\|\Lambda^2\mathbf{R}_{k,l}\|_{H^m}\lesssim1,\ &\text{for}\ m\in\{0,1\},\\\label{3app-decay}
|\mathbf{R}_{k,l}|+|\Lambda\mathbf{R}_{k,l}|+|\Lambda^2\mathbf{R}_{k,l}|\lesssim\langle x\rangle^{-2},\ &\text{for}\ x\in\mathbb{R}.
\end{align}
Moreover, the term on the right-hand side of \eqref{equ-3approxiamte} satisfies
\begin{align}\label{3app-regu-decay}
\|\mathbf{\Phi}_{\mathcal{P}}\|_{H^m}\lesssim\mathcal{O}(b^5+\lambda^2|\mathcal{P}|),\ |\nabla\mathbf{\Phi}_{\mathcal{P}}|\lesssim\mathcal{O}(b^5+\lambda^2|\mathcal{P}|)\langle x\rangle^{-2},
\end{align}
for $m\in\{0,1\}$ and $x\in\mathbb{R}$.
\end{lemma}
\begin{proof}
We recall that the definition of the linear operator
\begin{equation}\notag
L={\left[ \begin{array}{cc}
		L_+ & 0\\
		0 & L_{-}
		\end{array}	\right ]}
\end{equation}
acting on $L^2(\mathbb{R},\mathbb{R}^{2})$, where $L_+$ and $L_-$ denote the unbounded operators acting on $L^2(\mathbb{R},\ \mathbb{R}^{2})$ given by
\begin{align}\notag
	L_+=D+1-3Q^2,\ L_-=D+1-Q^2.
\end{align}
From \cite{Frank-lenzmann2013}, we have the key property that the kernel of $L$ is given by
\begin{equation}\notag
	\ker L=span\left\{{\left[ \begin{array}{c}
					\nabla Q\\ 0
				\end{array}
				\right ]},{\left[ \begin{array}{c}
					0\\Q
				\end{array}
				\right ]}\right\}.
\end{equation}
Note that the bounded inverse $L^{-1}=diag\{L_+^{-1},L_-^{-1}\}$ exists on the orthogonal complement $\{\ker L\}^{-1}=\{\nabla Q\}^{\bot}\bigoplus\{Q\}^{\bot}$.
				
$\mathbf{Step~1.}$ Determining the functions $\mathbf{R}_{k,l}$.
We discuss our ansatz for $Q_{\mathcal{P}}$ to solve \eqref{equ-3approxiamte} order by order. We inject \eqref{approximate-solution-in} into \eqref{equ-3approxiamte} and sort the terms of the same homogeneity.
		
\textbf{Order} $0$: Clearly, we have that
\begin{align*}
	DQ+Q-|Q|^2Q=0,
\end{align*}
since $Q$ being the ground state solution.
		
\textbf{Order} $1$: For $b$, from $k^{\prime}(0)=0$, we obtain
\begin{align*}
\begin{cases}
	L_+T_{1,0,}=0,\\
	L_-S_{1,0,}=\Lambda Q.
\end{cases}
\end{align*}
Due to the fact that $(\Lambda Q,Q)=0$, which can be easily seen by using the mass criticality.  Hence we can find a unique solution $S_{1,0}\perp\ker L_-$.

For $\lambda$, by the Taylor expansion, we have $T_{0,1}=S_{0,1}\equiv0$ since the function $k$ is even and $k^\prime(0)=0$.
		
\textbf{Order 2:} For $b^2$, we can obtain the following equation
\begin{align*}
\begin{cases}
	L_+T_{2,0}=\frac{1}{2}S_{1,0}-\Lambda S_{1,0}+S_{1,0,}^2Q,\\
	L_-S_{2,0}=0.
\end{cases}
\end{align*}
Since the kernel of $L_+$ is $\nabla Q$ and $S_{1,0}$, $Q$ are even. Hence we can find a unique solution $T_{2,0}$ to the above equation.

For $\lambda^2$, we have	
\begin{align*}
\begin{cases}
	L_+T_{0,2}=\frac{1}{2}K^{\prime\prime}(0)y^2Q^3,\\
	L_-S_{0,2}=0.
\end{cases}
\end{align*}
The solvability conditions read as
\begin{align*}
\left(\frac{1}{2}K^{\prime\prime}(0)y^2Q^3,\nabla Q\right)=0.
\end{align*}
However, this is obviously true, since $y^2$ and $Q$ are even functions.

For $\lambda b$, we can obtain that $T_{1,1}=S_{1,1}\equiv0$ since the assumption of the function $k$.

\textbf{Order 3.}  For $b^3$, we get
\begin{align*}
\begin{cases}
L_+T_{3,0}=0,\\
L_-S_{3,0}=-T_{2,0}+\Lambda T_{2,0}+2QT_{2,0}S_{1,0}+S_{1,0}^2S_{1,0}.
\end{cases}
\end{align*}
The solvability condition for $S_{3,0}$ is equivalent to
\begin{align}\label{approximate-3-in}
-(Q,T_{2,0})+(Q,\Lambda T_{2,0})+2(Q,QT_{2,0}S_{1,0})+(Q,S_{1,0}^2S_{1,0})=0.
\end{align}
To see \eqref{approximate-3-in} holds, we first note that
\begin{align*}
&\text{The right-hand side of above \eqref{approximate-3-in}}\\
=&-(Q,T_{2,0})-(\Lambda Q, T_{2,0})+2(T_{2,0},Q^2S_{1,0})+(Q,S_{1,0}^2S_{1,0})\\
=&-(Q,T_{2,0})-(L_-S_{1,0}, T_{2,0})+2(T_{2,0},Q^2S_{1,0})+(Q,S_{1,0}^2S_{1,0})\\
=&-(Q,T_{2,0})-(L_+S_{1,0}, T_{2,0})+(Q,S_{1,0}^2S_{1,0})\\
=&-(Q,T_{2,0})-\frac{1}{2}(S_{1,0},S_{1,0})+(S_{1,0},\Lambda S_{1,0})-(S_{1,0},S_{1,0}^2Q)+(Q,S_{1,0}^2S_{1,0})\\
=&-(Q,T_{2,0})-\frac{1}{2}(S_{1,0},S_{1,0}).
\end{align*}
Thus it remains to show that
\begin{align}\label{approximate-relation-in}
-(Q,T_{2,0})=\frac{1}{2}(S_{1,0},S_{1,0}).
\end{align}
Indeed, by using $L_+\Lambda Q=-Q$, we deduce
\begin{align}\label{approximate-4-in}
-(Q,T_{2,0})=&(\Lambda Q,\frac{1}{2}S_{1,0}-\Lambda S_{1,0}+S_{1,0}^2Q)\notag\\
=&\frac{1}{2}(L_-S_{1,0},S_{1,0})-(L_-S_{1,0},\Lambda S_{1,0})+(\Lambda Q,S_{1,0}^2Q)\notag\\
=&\frac{1}{2}(S_{1,0},DS_{1,0})+\frac{1}{2}(S_{1,0},S_{1,0})-\frac{1}{2}(S_{1,0},Q^2S_{1,0})\notag\\
&-(L_-S_{1,0},\Lambda S_{1,0})+(\Lambda Q,S_{1,0}^2Q).
\end{align}
Next, we apply $(L_-f,\Lambda f)=\frac{1}{2}(f,[L_-,\Lambda]f)$, which shows that
\begin{align}\label{approximate-5-in}
(L_-S_{1,0},\Lambda S_{1,0})=\frac{1}{2}(S_{1,0},[L_-,\Lambda]S_{1,0})
=&\frac{1}{2}(S_{1,0},[D,\Lambda]S_{1,0})-\frac{1}{2}(S_{1,0},[Q^2,\Lambda]S_{1,0})\notag\\
=&\frac{1}{2}(S_{1,0},DS_{1,0})+(S_{1,0},(x\cdot\nabla Q)QS_{1,0}).
\end{align}
Furthermore, we have the pointwise identity
\begin{align}\label{approximate-6-in}
-(x\cdot\nabla Q)Q+Q\Lambda Q=\frac{1}{2}Q^2.
\end{align}
Inserting \eqref{approximate-5-in} and \eqref{approximate-6-in} into \eqref{approximate-4-in}, we can obtain the desired relation \eqref{approximate-relation-in}, and thus the solvability condition \eqref{approximate-3-in} holds.

For $b^2\lambda$, we can easily obtain that $T_{2,1}=S_{2,1}\equiv0$ since the assumption of the function $k$.

\textbf{Order 4.} For $b^4$, we have
\begin{align*}
\begin{cases}
L_+T_{4,0}=\frac{3}{2}S_{3,0}-\Lambda S_{3,0}+T_{2,0}^2Q+2S_{1,0}S_{3,0}Q+2T_{2,0}QT_{2,0},\\
L_-S_{4,0}=0.
\end{cases}
\end{align*}
We notice that the right-hand side of above is a even function.  Hence there is a unique solution $T_{4,0}\perp\ker L_{+}$.
				
$\mathbf{Step~2}$. Regularity and decay bounds. By the similar argument as \cite{KLR2013}, we can obtain the regularity and decay bounds. Here we omit the details.
		
$\mathbf{Step~3}$. We mention that the bounds \eqref{3app-regu-decay} for the error term $\mathbf{\Phi}_{\mathcal{P}}$ follow from expanding $|Q_{\mathcal{P}}|^2Q_{\mathcal{P}}$ and using the regularity and decay bounds for the functions  $\{\mathbf{R}_{k,l}\}$. We omit the straightforward details. The proof of lemma \ref{lemma-3app} is now complete.
\end{proof}
		
We now turn to some key properties of the approximate blowup profile $Q_{\mathcal{P}}$ constructed in lemma \ref{lemma-3app}.
\begin{lemma}\label{lemma-3app-2}
The mass and the energy  of $Q_{\mathcal{P}}$ satisfy
\begin{align*}
\int|Q_{\mathcal{P}}|^2&=\int  Q^2+\mathcal{O}(b^4+\lambda^2),\\
E(Q_{\mathcal{P}})&=e_1b^2-\frac{\lambda^2}{8}\int K^{\prime\prime}(0)y^2Q^4+\mathcal{O}(b^4+\lambda^2).
\end{align*}
Here $e_1>0$ is the positive constant given by
\begin{align}\notag
e_1=\frac{1}{2}(L_{-}S_{1,0},S_{1,0}),\
\end{align}
where $S_{1,0}$ satisfy $L_{-}S_{1,0}=\Lambda Q$.
\end{lemma}
\begin{proof}
From the proof of lemma \ref{lemma-3app}, we have
\begin{align*}
\int|Q_{\mathcal{P}}|^2=&\int |Q+ibS_{1,0}+b^2T_{2,0}+\lambda^2T_{0,2}+ib^3S_{3,0}+b^4T_{4,0}|^2\\
=&\int Q^2+b^2(S_{1,0},S_{1,0})+2b^2(Q,T_{2,0})+\mathcal{O}(b^4+\lambda^2)\\
=&\int Q^2+\mathcal{O}(b^4+\lambda^2),
\end{align*}
where we used the relation \eqref{approximate-relation-in}
\begin{align*}
	(S_{1,0},S_{1,0})+2(Q,T_{2,0})=0.
\end{align*}	
To treat the expansion of the energy, we have
\begin{align*}
E(Q_{\mathcal{P}})=&\frac{1}{2} (Q_{\mathcal{P}},D Q_{\mathcal{P}})-\frac{1}{4}\int k(\lambda y)|Q_{\mathcal{P}}|^4\\
=&\frac{1}{2}(Q,DQ)-\frac{1}{4}\int k(\lambda y)|Q|^4+b^2\left\{\frac{1}{2}(S_{1,0},DS_{1,0})+(T_{2,0},D Q)\right\}\\
&-\frac{1}{2}b^2\int \left(k(\lambda y)|Q^2S_{1,0}^2+2Q^3T_{2,0}|\right)+\mathcal{O}(b^4+\lambda^2)\\
=&\frac{1}{2}(Q,DQ)-\frac{1}{4}\int |Q|^4+b^2\left\{\frac{1}{2}(S_{1,0},DS_{1,0})+(T_{2,0},D Q)\right\}\\
&-\frac{1}{2}b^2\int k(\lambda y)\left(|Q^2S_{1,0}^2+2Q^3T_{2,0}|\right)-\frac{1}{4}\int\left(k(\lambda y)-1\right)|Q|^4\\
&+\mathcal{O}(b^4+\lambda^2)\\
=&b^2\left\{\frac{1}{2}(S_{1,0},DS_{1,0})+(T_{2,0},DQ)-\frac{1}{2}(Q,QS_{1,0}^2)-(T_{2,0},Q^3)\right\}\}\\
&-\frac{1}{4}\int\left( k(\lambda y)-1\right)|Q|^4+\mathcal{O}(b^4+\lambda^2)\\
=&b^2\left\{\frac{1}{2}(S_{1,0},DS_{1,0})-(T_{2,0},Q)-\frac{1}{2}(Q,QS_{1,0}^2)\right\}\\
&-\frac{1}{4}\int\left(k(\lambda y)-1\right)|Q|^4+\mathcal{O}(b^4+\lambda^2)\\
=&b^2\left\{\frac{1}{2}(S_{1,0},DS_{1,0})+\frac{1}{2}(S_{1,0},S_{1,0})-\frac{1}{2}(Q,QS_{1,0}^2)\right\}\\
&-\frac{1}{4}\int\left(k(\lambda y)-1\right)|Q|^4+\mathcal{O}(b^4+\lambda^2)\\
=&\frac{1}{2}b^2(L_-S_{1,0},S_{1,0})-\frac{\lambda^2}{8}\int k^{\prime\prime}(0)y^2Q^4+\mathcal{O}(b^4+\lambda^2).
\end{align*}
Here we used $\tilde{E}(Q)=\frac{1}{2}(Q,DQ)-\frac{1}{4}\int |Q|^4=0$, $(S_{1,0},S_{1,0})+2(Q,T_{2,0})=0$ and expanding the inhomogenity $k(\lambda y)$ near $0$.
The proof of lemma \ref{lemma-3app-2} is now complete.
\end{proof}

\section{Modulation Estimates}\label{section-mod-estimate}
	
Let $u\in H^{1/2}(\mathbb{R})$ be a solution of \eqref{equ-1-hf-in} on some time interval $[t_0,t_1]$ with $t_1<0$. Assume that $u(t)$ admits a geometrical decomposition of the form
\begin{align}\label{mod-decomposition}
u(t,x)=\frac{1}{\lambda^{\frac{1}{2}}(t)}[Q_{\mathcal{P}(t)}+\epsilon]\left(s,\frac{x}{\lambda(t)}\right)e^{i\gamma(t)},\ \ \frac{ds}{dt}=\frac{1}{\lambda(t)},
\end{align}
with $\mathcal{P}(t)=(b(t),\lambda(t))$, and we impose the uniform smallness bound
\begin{align}
	|\mathcal{P}|+\|\epsilon\|_{H^{1/2}}^2\lesssim\lambda(t)\ll1.
\end{align}
Furthermore, we assume that $u(t)$ has almost critical mass in the sense that
\begin{align}\label{mod-mass-assume}
\left|\int|u(t)|^2-\int Q^2\right|\lesssim\lambda^2(t),\ \ \forall t\in[t_0,t_1].
\end{align}
To fix the modulation parameters $\{b(t),\lambda(t),\gamma(t)\}$ uniquely, we impose the following orthogonality conditions on $\epsilon=\epsilon_1+i\epsilon_2$ as follows:
\begin{align}\label{mod-orthogonality-condition}
\begin{cases}
(\epsilon_2,\Lambda Q_{1\mathcal{P}})-(\epsilon_1,\Lambda Q_{2\mathcal{P}})=0,\\
(\epsilon_2,\partial_bQ_{1\mathcal{P}})-(\epsilon_1,\partial_bQ_{2\mathcal{P}})=0,\\
(\epsilon_2,\rho_1)-(\epsilon_1,\rho_2)=0,
\end{cases}
\end{align}
the function $\rho=\rho_1+i\rho_2$ is defined by
\begin{align}\label{mod-definition-rho}
	L_{+}\rho_1=S_{1,0},\,\, L_{-}\rho_2=bS_{1,0}\rho_1+b\Lambda\rho_1-2bT_{2,0},
\end{align}
where $S_{1,0}$ and $T_{2,0}$ are the functions introduced in the proof of lemma \ref{lemma-3app}. Note that $L_{+}^{-1}$ exists on $L^2_{rad}(\mathbb{R})$ and thus $\rho_1$ is well-defined. Moreover, it is easy to see that the right-hand side in the equation for $\rho_2$ is orthogonality to $Q$. Indeed
\begin{align*}
(Q,S_{1,0}\rho_1+\Lambda\rho_1-2T_{2,0})
&=(QS_{1,0},\rho_1)-(\Lambda Q,\rho_1)-2(Q,T_{2,0})\\
&=(QS_{1,0},\rho_1)-(S_{1,0},L_{-}\rho_1)+(S_{1,0},S_{1,0})\\
&=-(S_{1,0},L_{+}\rho_1)+(S_{1,0},S_{1,0})=0,
\end{align*}
using that $(S_{1,0},S_{1,0})=-2(T_{2,0},Q)$, see \eqref{approximate-relation-in}, and the definition of $\rho_1$.  Hence $\rho_2$ is well-defined.
	
In the conditions \eqref{mod-orthogonality-condition}, we use the notation
\begin{align}\notag
		Q_{\mathcal{P}}=Q_{1\mathcal{P}}+iQ_{2\mathcal{P}}.
\end{align}

By the standard arguments as \cite{KLR2013,Georgiev-Li-2022-CPDE,Georgiev-Li-2021-JFA},  the orthogonality conditions \eqref{mod-orthogonality-condition} imply that the modulation parameters $\{b(t),\lambda(t),\gamma(t)\}$ are uniquely determined, provided that $\epsilon=\epsilon_1+i\epsilon_2\in H^{1/2}(\mathbb{R})$ is sufficiently small. Moreover, it follows from the standard arguments that $\{b(t),\lambda(t),\gamma(t)\}$ are $C^1$-functions.
	
If we insert the decomposition \eqref{mod-decomposition} into \eqref{equ-1-hf-in}, we obtain the following system
\begin{align}\label{mod-system-1}
&\left(b_s+\frac{1}{2}b^2\right)\partial_bQ_{1\mathcal{P}}+\lambda\left(\frac{\lambda_s}{\lambda}+b\right)\partial_{\lambda}Q_{1\mathcal{P}}+\partial_s\epsilon_1
-M_{-}(\epsilon)+b\Lambda\epsilon_1\notag\\
=&\left(\frac{\lambda_s}{\lambda}+b\right)(\Lambda Q_{1\mathcal{P}}+\Lambda\epsilon_1)
+\tilde{\gamma}_s(Q_{2\mathcal{P}}+\epsilon_2)+\Im(\Phi_{\mathcal{P}})-R_2(\epsilon),\\\label{mod-system-2}
&\left(b_s+\frac{1}{2}b^2\right)\partial_bQ_{2\mathcal{P}}+\lambda\left(\frac{\lambda_s}{\lambda}+b\right)\partial_{\lambda}Q_{2\mathcal{P}}
+\partial_s\epsilon_2
-M_{+}(\epsilon)+b\Lambda\epsilon_2\notag\\
=&\left(\frac{\lambda_s}{\lambda}+b\right)(\Lambda Q_{2\mathcal{P}}+\Lambda\epsilon_2)
+\tilde{\gamma}_s(Q_{2\mathcal{P}}+\epsilon_1)+\Re(\Phi_{\mathcal{P}})-R_1(\epsilon).
\end{align}
Here $\Phi_{\mathcal{P}}$ denotes the error term from lemma \ref{lemma-3app}, and $M=(M_{+},M_{-})$ are the small deformations of the linearized operator $L=(L_{+},L_{-})$ given by
\begin{align}\label{mod-define-M1}
M_{+}(\epsilon)=&D\epsilon_1+\epsilon_1-k(\lambda y)\left(|Q_{\mathcal{P}}|^2\epsilon_1
+2Q_{1\mathcal{P}}^2\epsilon_1+2Q_{1\mathcal{P}}Q_{2\mathcal{P}}\epsilon_2\right)\\\label{mod-define-M2}
M_{-}(\epsilon)=&D\epsilon_2+\epsilon_2
-k(\lambda y)\left(|Q_{\mathcal{P}}|^2\epsilon_2
+2Q_{1\mathcal{P}}^2\epsilon_2+2Q_{1\mathcal{P}}Q_{2\mathcal{P}}\epsilon_1\right).
\end{align}
And $R_1(\epsilon)$, $R_2(\epsilon)$ are the high order terms about $\epsilon$.
\begin{align*}
R_1(\epsilon)=&k(\lambda y)\left(3Q_{1\mathcal{P}}\epsilon_1^2+2Q_{2\mathcal{P}}\epsilon_1\epsilon_2+Q_{1\mathcal{P}}\epsilon_2^2+|\epsilon|^2\epsilon_1\right),\\
R_2(\epsilon)=&k(\lambda y)\left(3Q_{2\mathcal{P}}\epsilon_2^2+2Q_{1\mathcal{P}}\epsilon_1\epsilon_2+Q_{2\mathcal{P}}\epsilon_1^2+|\epsilon|^2\epsilon_2\right).
\end{align*}
	
We have the following energy type bound.
\begin{lemma}(\textbf{Preliminary estimate on the decomposition}.)\label{lemma-mod-1}
For $t\in[t_0,t_1]$ with $t_1<0$, there holds that
\begin{align}\label{mod-1-energy-estimate}
b^2+\|\epsilon\|_{H^{1/2}}^2\lesssim\lambda|E_0|+\mathcal{O}(\lambda^2+b^4).
\end{align}
Here $E_0=E(u_0)$ denote the conserved energy  of $u=u(t,x)$.

Define the vector-valued function
\begin{align}
	Mod(t):=\left(b_s+\frac{b^2}{2},\tilde{\gamma}_s,\frac{\lambda_s}{\lambda}+b\right).
\end{align}
Then, for $t\in[t_0,t_1]$, we have the bound
\begin{align}
	|Mod(t)|\lesssim \lambda^2+b^4+\mathcal{P}\|\epsilon\|_{L^2}+\|\epsilon\|_{L^2}^2+\|\epsilon\|_{H^{1/2}}^3.
\end{align}
Furthermore, we have the improved bound
\begin{align*}
	\left|\frac{\lambda_s}{\lambda}+b\right|\lesssim b^5+\mathcal{P}\|\epsilon\|_{L^2}+\|\epsilon\|_{L^2}^2+\|\epsilon\|_{H^{1/2}}^2.
\end{align*}
\end{lemma}
\begin{proof}
We divide the proof into the following steps.\\
\textbf{Step 1.} Conservation of $L^2$-norm and energy.

By the conservation of $L^2$-mass and lemma \ref{lemma-3app-2}, we find that
\begin{align}\notag
\int|u|^2=\int|Q_{\mathcal{P}}+\epsilon|^2=\int|Q|^2+2\Re(\epsilon,Q_{\mathcal{P}})+\int|\epsilon|^2+\mathcal{O}(\lambda^2+b^4).
\end{align}
By assumption \eqref{mod-mass-assume}, this implies
\begin{align}\label{mod-mass}
2\Re(\epsilon,Q_{\mathcal{P}})+\int |\epsilon|^2=\mathcal{O}(\lambda^2+b^4).
\end{align}
Next, we recall that $v=Q_{\mathcal{P}}+\epsilon$ and the assumed form of $u=u(t,x)$. We expanding the nonlinear term:
\begin{align*}
|v|^4=&|Q_{\mathcal{P}}|^4+4\Re(\epsilon\overline{|Q_{\mathcal{P}}|^2Q_{\mathcal{P}}})+4\Re(|\epsilon|^2\epsilon\bar{Q}_{\mathcal{P}})\\
&+2|Q_{\mathcal{P}}|^2\left[\left(1+\frac{2Q_{1\mathcal{P}}^2}{|Q_{\mathcal{P}}|^2}\right)\epsilon_1^2+4\frac{Q_{1\mathcal{P}}Q_{2\mathcal{P}}}{|Q_{\mathcal{P}}|^2}\epsilon_1\epsilon_2+\left(1+\frac{2Q_{2\mathcal{P}}^2}{|Q_{\mathcal{P}}|^2}\right)\epsilon_2^2\right].
\end{align*}
We now inject the value of $\tilde{E}(Q_{\mathcal{P}})$ given by lemma \ref{lemma-3app-2} and the estimate the cubic and higher nonlinear terms, using the Gagliardo-Nirenberg inequality and the priori smallness \eqref{mod-mass-assume} to derive:
\begin{align*}
&\frac{1}{2}\int|D^{\frac{1}{2}}(Q_{\mathcal{P}}+\epsilon)|^2-\frac{1}{4}\int k(\lambda y)|Q_{\mathcal{P}}+\epsilon|^4\\
=&\tilde{E}(Q_{\mathcal{P}})+\Re\left(\epsilon,\overline{DQ_{\mathcal{P}}-k(\lambda y)|Q_{\mathcal{P}}|^2Q_{\mathcal{P}}}\right)\\
&+\frac{1}{2}\int|D^{\frac{1}{2}}\epsilon|^2-\frac{1}{2}\int 2|Q_{\mathcal{P}}|^2\left[\left(1+\frac{2Q_{1\mathcal{P}}^2}{|Q_{\mathcal{P}}|^2}\right)\epsilon_1^2+4\frac{Q_{1\mathcal{P}}Q_{2\mathcal{P}}}{|Q_{\mathcal{P}}|^2}\epsilon_1\epsilon_2+\left(1+\frac{2Q_{2\mathcal{P}}^2}{|Q_{\mathcal{P}}|^2}\right)\epsilon_2^2\right]\\
&+\mathcal{O}(\|\epsilon\|_{H^{1/2}}^3+\mathcal{P}^2\|\epsilon\|_{H^{1/2}}^2)\\
=&e_1b^2-\frac{\lambda^2}{8}\int k^{\prime\prime}(0)y^2Q^4+\Re\left(\epsilon,\overline{DQ_{\mathcal{P}}-k(\lambda y)|Q_{\mathcal{P}}|^2Q_{\mathcal{P}}}\right)\\
&+\frac{1}{2}\int|D^{\frac{1}{2}}\epsilon|^2-\frac{1}{2}\int 2|Q_{\mathcal{P}}|^2\left[\left(1+\frac{2Q_{1\mathcal{P}}^2}{|Q_{\mathcal{P}}|^2}\right)\epsilon_1^2+4\frac{Q_{1\mathcal{P}}Q_{2\mathcal{P}}}{|Q_{\mathcal{P}}|^2}\epsilon_1\epsilon_2+\left(1+\frac{2Q_{2\mathcal{P}}^2}{|Q_{\mathcal{P}}|^2}\right)\epsilon_2^2\right]\\
&+\mathcal{O}(\|\epsilon\|_{H^{1/2}}^3+\mathcal{P}^4).
\end{align*}
Hence, we have
\begin{align*}
\lambda E_0=&e_1b^2-\frac{\lambda^2}{8}\int k^{\prime\prime}(0)y^2Q^4+\Re\left(\epsilon,\overline{DQ_{\mathcal{P}}-k(\lambda y)|Q_{\mathcal{P}}|^2Q_{\mathcal{P}}}\right)\\
&+\frac{1}{2}\int|D^{\frac{1}{2}}\epsilon|^2-\frac{1}{2}\int 2|Q_{\mathcal{P}}|^2\left[\left(1+\frac{2Q_{1\mathcal{P}}^2}{|Q_{\mathcal{P}}|^2}\right)\epsilon_1^2+4\frac{Q_{1\mathcal{P}}Q_{2\mathcal{P}}}{|Q_{\mathcal{P}}|^2}\epsilon_1\epsilon_2+\left(1+\frac{2Q_{2\mathcal{P}}^2}{|Q_{\mathcal{P}}|^2}\right)\epsilon_2^2\right]\\
&+\mathcal{O}(\|\epsilon\|_{H^{1/2}}^3+\mathcal{P}^4).
\end{align*}
\textbf{Step 2.} Coercivity of the linearized energy and the proof of \eqref{mod-1-energy-estimate}. Combining the conservation of mass \eqref{mod-mass} and the above energy conservation, we deduce that
\begin{align}\label{mod-1}
&b^2e_1+\Re\left(\epsilon,\overline{Q_{\mathcal{P}}+DQ_{\mathcal{P}}-k(\lambda y)|Q_{\mathcal{P}}|^2Q_{\mathcal{P}}}\right)+\frac{1}{2}\int|\epsilon|^2\notag\\
&+\frac{1}{2}\int|D^{\frac{1}{2}}\epsilon|^2-\frac{1}{2}\int 2|Q_{\mathcal{P}}|^2\left[\left(1+\frac{2Q_{1\mathcal{P}}^2}{|Q_{\mathcal{P}}|^2}\right)\epsilon_1^2+4\frac{Q_{1\mathcal{P}}Q_{2\mathcal{P}}}{|Q_{\mathcal{P}}|^2}\epsilon_1\epsilon_2+\left(1+\frac{2Q_{2\mathcal{P}}^2}{|Q_{\mathcal{P}}|^2}\right)\epsilon_2^2\right]\notag\\
=&\lambda E_0+\frac{\lambda^2}{8}\int k^{\prime\prime}(0)y^2Q^4+\mathcal{O}\left(\|\epsilon\|_{H^{1/2}}^3+\mathcal{P}^4+b^4+\lambda^2\right).
\end{align}
The remaining linear term is degenerate from \eqref{equ-3approxiamte}:
\begin{align*}
Q_{\mathcal{P}}+DQ_{\mathcal{P}}-k(\lambda(t)y)|Q_{\mathcal{P}}|^2Q_{\mathcal{P}}=-b\Lambda Q_{\mathcal{P}}+\mathcal{O}(\mathcal{P}^2),
\end{align*}
and thus using the orthogonality condition \eqref{mod-orthogonality-condition}, we have
\begin{align}\label{mod-2}
&\Re\left(\epsilon,\overline{Q_{\mathcal{P}}+DQ_{\mathcal{P}}-k(\lambda y)|Q_{\mathcal{P}}|^2Q_{\mathcal{P}}}\right)\notag\\
=&b\Im(\epsilon,\overline{\Lambda Q_{\mathcal{P}}})+\mathcal{O}(\mathcal{P}^2\|\epsilon\|_{L^2})=\mathcal{O}(\mathcal{P}^2\|\epsilon\|_{L^2}).
\end{align}
From \eqref{mod-1} and \eqref{mod-2}, we deduce that
\begin{align}\label{mod-3}
&b^2e_1+\frac{1}{2}\int|\epsilon|^2+\frac{1}{2}\int|D^{\frac{1}{2}}\epsilon|^2\notag\\
-&\frac{1}{2}\int 2|Q_{\mathcal{P}}|^2\left[\left(1+\frac{2Q_{1\mathcal{P}}^2}{|Q_{\mathcal{P}}|^2}\right)\epsilon_1^2+4\frac{Q_{1\mathcal{P}}Q_{2\mathcal{P}}}{|Q_{\mathcal{P}}|^2}\epsilon_1\epsilon_2+\left(1+\frac{2Q_{2\mathcal{P}}^2}{|Q_{\mathcal{P}}|^2}\right)\epsilon_2^2\right]\notag\\
=&\lambda E_0+\frac{\lambda^2}{8}\int k^{\prime\prime}(0)y^2Q^4+\mathcal{O}(\|\epsilon\|_{H^{1/2}}^3+\|\mathcal{P}^4+b^4+\lambda^2).
\end{align}
We now observe from the proximity of $Q_{\mathcal{P}}$ to $Q$ ensured by the priori smallness \eqref{mod-mass-assume} that the quadratic form in the left-hand side of \eqref{mod-3} is a small deformation of the linearized operator close to $Q$. Hence, we have
\begin{align}\label{mod-4}
&b^2e_1+\frac{1}{2}\left[(L_+\epsilon_1,\epsilon_1)+(L_-\epsilon_2,\epsilon_2)\right]\notag\\
=&\lambda E_0+\frac{\lambda^2}{8}\int k^{\prime\prime}(0)y^2Q^4+\mathcal{O}(\|\epsilon\|_{H^{1/2}}^3+b^4+\lambda^2)+o(\|\epsilon\|_{H^{1/2}}^2).
\end{align}
We now recall the following coercivity property of the linearized energy, which is a consequence of the variational characterization of $Q$,
\begin{align}\label{coercivity-estimate}
(L_+\epsilon_1,\epsilon_1)&+(L_-\epsilon_2,\epsilon_2)\geq C_0\|\epsilon\|_{H^{1/2}}
-\frac{1}{C_0}\{(\epsilon_1,Q)^2+(\epsilon_1,S_{1,0})^2+(\epsilon_2,\rho_1)^2\}.
\end{align}
with some constant $C_0>0$. By the similar argument as \cite{KLR2013,MerleR2006}, we can obtain the coercivity estimate. Note that the orthogonality conditions \eqref{mod-orthogonality-condition} imply that
\begin{align*}
    (\epsilon[_1,S_{1,0})^2=\mathcal{O}(\mathcal{P}^2\|\epsilon\|_{L^2}^2),\,\,(\epsilon_2,\rho_1)^2=\mathcal{O}(\mathcal{P}^2\|\epsilon\|_{L^2}^2).
\end{align*}
Furthermore, using the orthogonality conditions \eqref{mod-orthogonality-condition} together with the degeneracy inherited from \eqref{mod-mass}, we have
\begin{align*}
|(\epsilon_1,Q)|^2=o(\|\epsilon\|_{H^{1/2}})+\mathcal{O}(b^4+\lambda^2).
\end{align*}
Then combining these bounds with \eqref{coercivity-estimate}, we deduce that
\begin{align*}
(L_+\epsilon_1,\epsilon_1)&+(L_-\epsilon_2,\epsilon_2)\geq\frac{C_0}{2}\|\epsilon\|_{H^{1/2}}^2+\mathcal{O}(b^4+\lambda^2).
\end{align*}
Injecting this into \eqref{mod-4}, we can obtain \eqref{mod-1-energy-estimate}.
	
\textbf{Step~3}. Inner products. We compute the inner products needed to compute the law of the parameters from the $Q_{\mathcal{P}}$ equation \eqref{equ-3approxiamte}, where $M_1$ and $M_2$ are given by \eqref{mod-define-M1} and \eqref{mod-define-M2}, respectively. The following estimates hold.
\begin{align}\label{mod-2-1}
&(M_-(\epsilon)-b\Lambda \epsilon_1,\Lambda Q_{2\mathcal{P}})+(M_+(\epsilon)+b\Lambda\epsilon_2,\Lambda Q_{1\mathcal{P}})=-\Re(\epsilon,Q_{\mathcal{P}})+\mathcal{O}(\mathcal{P}^2\|\epsilon\|_{L^2}),\\\label{mod-2-2}
&(M_-(\epsilon)-b\Lambda \epsilon_1,\partial_b Q_{2\mathcal{P}})+(M_+(\epsilon)+b\Lambda\epsilon_2,\partial_b Q_{1\mathcal{P}})=\mathcal{O}(\mathcal{P}^2\|\epsilon\|_{L^2}),\\ \label{mod-2-3}
&(M_-(\epsilon)-b\Lambda \epsilon_1,\rho_2)+(M_+(\epsilon)+b\Lambda\epsilon_2,\rho_1)=\mathcal{O}(\mathcal{P}^2\|\epsilon\|_{L^2}).
\end{align}
To prove \eqref{mod-2-1}-\eqref{mod-2-3}, we can use the similar argument as \cite{Raphael2011-Jams,Rephael-2007-poincare,KLR2013}, here we omit the details.
	
\textbf{Step~4.} Simplification of the equations. To compute the modulation equations driving the geometrical parameters, we first  simplify the equations \eqref{mod-system-1} and \eqref{mod-system-2}. Using the explicit construction of $Q_{\mathcal{P}}$, we obtain
\begin{align}\label{mod-equ-simple-1}
&(b_s+\frac{1}{2}b^2)\partial_bQ_{1\mathcal{P}}+\partial_s\epsilon_1-M_{-}(\epsilon)+b\Lambda\epsilon_1\notag\\
=&\left(\frac{\lambda_s}{\lambda}+b\right)(\Lambda Q_{1\mathcal{P}}+\Lambda\epsilon_1)
+\tilde{\gamma}_s(Q_{2\mathcal{P}}+\epsilon_2)+\Im(\tilde{\Phi}_{\mathcal{P}})-R_2(\epsilon),\\\label{mod-equ-simple-2}
&(b_s+\frac{1}{2}b^2)\partial_bQ_{2\mathcal{P}}+\partial_s\epsilon_2-M_{+}(\epsilon)+b\Lambda\epsilon_2\notag\\
=&\left(\frac{\lambda_s}{\lambda}+b\right)(\Lambda Q_{2\mathcal{P}}+\Lambda\epsilon_2)
+\tilde{\gamma}_s(Q_{2\mathcal{P}}+\epsilon_1)-\Re(\tilde{\Phi}_{\mathcal{P}})+R_1(\epsilon),
\end{align}
where $M_+$ and $M_-$ defined by \eqref{mod-define-M1} and \eqref{mod-define-M2}, respectively, and the remainder term ${\Phi}_{\mathcal{P}}$  has the following from from lemma \ref{lemma-3app}.

\textbf{Step 5.} The law of $b$. We take the inner product of the equation \eqref{mod-equ-simple-1} of $\epsilon_1$ with $-\Lambda Q_{2\mathcal{P}}$ and we sum it with the inner product of equation \eqref{mod-equ-simple-2} of $\epsilon_{2}$ with $\Lambda Q_{1\mathcal{P}}$. We obtain after integrating by parts:
\begin{align*}
	-\left(b_s+\frac{b^2}{2}\right)\left((L_-S_{1,0},S_{1,0})+\mathcal{O}(\mathcal{P}^2)\right)=&\Re(\epsilon,Q_{\mathcal{P}})+(R_2(\epsilon),\Lambda Q_{1\mathcal{P}})+(R_1(\epsilon),\Lambda Q_{2\mathcal{P}})\notag\\
	&-(\Im(\tilde{\Phi}_{\mathcal{P}}),\Lambda Q_{2\mathcal{P}})+(\Re(\tilde{\Phi}_{\mathcal{P}}),\Lambda Q_{1\mathcal{P}})\notag\\
	&+\mathcal{O}\left((\mathcal{P}^2+|Mod(t)|)(\|\epsilon\|_{L^2}+\mathcal{P})\right).
\end{align*}
Now, by using that
\begin{align}
	2\Re(\epsilon,Q_{\mathcal{P}})=-\int |\epsilon|^2+\mathcal{O}(b^4+\lambda^2),
\end{align}
We deduce that
\begin{align*}
	-\left(b_s+\frac{b^2}{2}\right)\left((L_-S_{1,0},S_{1,0})+\mathcal{O}(\mathcal{P}^2)\right)
	=&-\int|\epsilon|^2+(R_2(\epsilon),\Lambda Q_{1\mathcal{P}})+(R_1(\epsilon),\Lambda Q_{2\mathcal{P}})\\
	&+\mathcal{O}\left((\mathcal{P}+|Mod(t)|)(\|\epsilon\|_{L^2}+\mathcal{P})+\lambda^2+b^4\right).
\end{align*}

\textbf{Step 6.} The law of $\lambda$. By projecting \eqref{mod-equ-simple-1} and \eqref{mod-equ-simple-2} onto $-\partial_bQ_{2\mathcal{P}}$ and $\partial_bQ_{1\mathcal{P}}$, respectively. We obtain
\begin{align*}
	\left(\frac{\lambda_s}{\lambda}+b\right)(2e_1+\mathcal{P}^2)=&(R_2(\epsilon),\partial_bQ_{1\mathcal{P}})+(R_1(\epsilon),\partial_bQ_{2\mathcal{P}})\\
	&+\mathcal{O}\left((\mathcal{P}^2+|Mod(t)|)(\|\epsilon\|_{L^2}+\mathcal{P}^2)+b^5\right).
\end{align*}
Here we used that $(Q_{2\mathcal{P}},\partial_bQ_{2\mathcal{P}})+(Q_{1\mathcal{P}},\partial_bQ_{1\mathcal{P}})=b(S_{1,0},S_{1,0})+2b(Q,T_{2,0})+\mathcal{O}(\mathcal{P}^2)=\mathcal{O}(\mathcal{P}^2)$, since $(S_{1,0},S_{1,0})+2(Q,T_{2,0})=0$.

\textbf{Step 7.} Law for $\tilde{\gamma}$.

We take the inner product of the equation \eqref{mod-equ-simple-1} of $\epsilon_1$ with $-\rho_2$ and we sum it with the inner product of equation \eqref{mod-equ-simple-2} of $\epsilon_{2}$ with $\rho_1$. We obtain after integrating by parts:
\begin{align*}
	\tilde{\gamma_s}\left((Q,\rho_1)+\mathcal{O}(\mathcal{P}^2)\right)=&-\left(b_s+\frac{b^2}{2}\right)\left((S_{1,0},\rho_1)+\mathcal{O}(\mathcal{P}^2)\right)+\left(\frac{\lambda_s}{\lambda}+b\right)(\mathcal{O}(\mathcal{P}))\\
	&+(R_2(\epsilon),\rho_1)+(R_1(\epsilon),\rho_2)+\mathcal{O}\left((\mathcal{P}^2+|Mod(t)|)\|\epsilon\|_{L^2}+b^5\right).
\end{align*}
Note that $(Q,\rho_1)=(L_-S_{1,0},S_{1,0})=2e_1$, which follows from $L_+\Lambda Q=-Q$ and the definition of $\rho_1$.

\textbf{Step 8.} Conclusion. Putting together Step 5, 6, 7 and estimate the nonlinear interaction terms in  $\epsilon$ by using Sobolev embedding yields:
\begin{align}
	(A+B)Mod(t)=\mathcal{O}\left((\mathcal{P}^2+|Mod(t)|)\|\epsilon\|_{L^2}+\|\epsilon\|_{L^2}^2+\|\epsilon\|_{H^{1/2}}^3+\lambda^2+b^4\right).
\end{align}
Here $A=\mathcal{O}(1)$ is in invertible $3\times3$ matrix, whereas $B=\mathcal{O}(\mathcal{P})$ is some $3\times3$ matrix that is polynomial in $\mathcal{P}=(b,\lambda)$. For $|\mathcal{P}|\ll1$, we can thus invert $A+B$ by Taylor expansion and derive the estimate for $Mod(t)$ stated in this lemma.

Finally, we deduce the improved bound for $\left|\frac{\lambda_s}{\lambda}+b\right|$, by recalling the estimate derived in Step 6.
Now we complete the proof of this lemma.
\end{proof}

\section{Refined Energy bounds}\label{section-refined-energy}
In this section, we establish a refined energy estimate, which will be a key ingredient in the compactness argument to construct ground state mass blowup solutions.
	
Let $u=u(t,x)$ be a solution \eqref{equ-1-hf-in} on the time interval $[t_0,0)$ and suppose that $\tilde{Q}$ is an approximate solution to \eqref{equ-1-hf-in} such that
\begin{equation}\label{equ-refine-approximate-1}
i\tilde{Q}_t-D\tilde{Q}+k(x)|\tilde{Q}|^2\tilde{Q}=\psi,
\end{equation}
with the priori bounds
\begin{align}\label{energy-priori-1-in}
\|\tilde{Q}\|_2\lesssim 1,\ \|D^{\frac{1}{2}}\tilde{Q}\|_2\lesssim \frac{1}{\lambda^{\frac{1}{2}}},\ \|\nabla \tilde{Q}\|_2\lesssim \frac{1}{\lambda}.
\end{align}
We then decompose $u=\tilde{Q}+\tilde{\epsilon}$, and hence $\tilde{\epsilon}$ satisfies
\begin{equation}\label{equ-app2-hf-in}
i\tilde{\epsilon}_t-D\tilde{\epsilon}+k(x)(|u|^2u-|\tilde{Q}|^2\tilde{Q})=-\psi,
\end{equation}
and we assume the priori bounds
\begin{align}\label{energy-priori-2-in}
\|D^{\frac{1}{2}+\delta}\tilde{\epsilon}\|_2\lesssim1,\,\,\|D^{\frac{1}{2}}\tilde{\epsilon}\|_{L^2}\lesssim\lambda^{\frac{1}{2}},\,\,\|\tilde{\epsilon}\|_2\lesssim \lambda,
\end{align}
as well as
\begin{align}\label{energy-priori-3-in}
|\lambda_t+b|\lesssim\lambda^2,\,\, b\lesssim\lambda^{\frac{1}{2}},\,\, |b_t|\lesssim1.
\end{align}
We let $A>0$ be a large enough constant, which will be chosen later, and let
$\phi:\mathbb{R}\rightarrow\mathbb{R}$ be a smooth and even cutoff function with
\begin{align}\label{energy-cutoff-function}
	\phi^\prime(x)=\begin{cases}x\ \ &\text{for}\ \ 0\leq x\leq1,\\
			3-e^{-|x|}\ &\text{for}\ x\geq2,
	\end{cases}
\end{align}
and the convexity condition
\begin{align}
	\phi^{\prime\prime}(x)\geq0\ \text{for}\ x\geq0.
\end{align}
Furthermore, we denote
\begin{align}\notag
F(u)=\frac{1}{4}|u|^{4},\ f(u)=|u|^2u,\ F'(u)\cdot h=\Re(f(u)\bar{h}).
\end{align}	
We have the following generalized energy estimate.
\begin{lemma}\label{lemma-refine-energy-in}(\textbf{Localized energy/virial estimate}.) Let
\begin{align}\label{refine-energy-def}
J_A(u):=&\frac{1}{2}\int|D^{\frac{1}{2}}\tilde{\epsilon}|^2
+\frac{1}{2}\int\frac{|\tilde{\epsilon}|^2}{\lambda}-\int k(x)[F(u)-F(\tilde{Q})-F'(\tilde{Q})\cdot\tilde{\epsilon}]\notag\\
&+\frac{b}{2}\Im\left(\int A\nabla\phi\left(\frac{x}{A\lambda}\right)\cdot\nabla\tilde{\epsilon}\bar{\tilde{\epsilon}}\right).
\end{align}
Then the following holds:
\begin{align}\label{refine-energy-estimate}
\frac{d J_{A}}{dt}=&-\frac{1}{\lambda}\Im\left(\int k(x)w^2\bar{\tilde{\epsilon}}^2 \right)-\Re\left(\int k(x)\tilde{Q}_t\overline{(2|\tilde{\epsilon}|^2\tilde{Q}+\tilde{\epsilon}^2\bar{\tilde{Q}}})\right)\notag\\
&+\frac{b}{2\lambda}\int\frac{|\tilde{\epsilon}|^2}{\lambda}+\frac{b}{2\lambda}\int_{s=0}^{+\infty}\sqrt{s}\int \Delta\phi\left(\frac{x}{A\lambda}\right)|\nabla \tilde{\epsilon}_s|^2dxds\notag\\
&-\frac{1}{8}\frac{b}{A^2\lambda^3}\int_{s=0}^{\infty}\sqrt{s}\int \Delta^2\phi\left(\frac{x}{A\lambda}\right)|\tilde{\epsilon}_s|^2dxds\notag\\
&+b\Re\left(\int A\nabla\phi\left(\frac{x}{A\lambda}\right)k(x)(2|\tilde{\epsilon}|^2\tilde{Q}+\tilde{\epsilon}^2\bar{\tilde{Q}})\cdot\overline{\nabla \tilde{Q}}\right)\notag\\
&+\Im\Bigg(\int \bigg[-D\psi-\frac{\psi}{\lambda}+k(x)(2|\tilde{Q}|^2\psi-\tilde{Q}^2\bar{\psi})+ibA\nabla\phi\left(\frac{x}{A\lambda}\right)\cdot\nabla\psi\notag\\
&+i\frac{b}{2\lambda}\Delta\phi\left(\frac{x}{A\lambda}\right)\psi\bigg]\bar{\tilde{\epsilon}}\Bigg)\notag\\
&+\mathcal{O}\left(\lambda\|\psi\|_{L^2}^2+\lambda^{-1}\|\tilde{\epsilon}\|_{L^2}+\log^{\frac{1}{2}}\left(2+\|\tilde{\epsilon}\|_{H^{1/2}}^{-1}\right)\|\tilde{\epsilon}\|_{H^{1/2}}^2\right).
\end{align}
Here we denote $\tilde{\epsilon}_s=\sqrt{\frac{2}{\pi}}\frac{1}{-\Delta+s}\tilde{\epsilon}$ with $s>0$.
\end{lemma}
\begin{proof}
To prove this lemma, we can use the similar arguments that can be found  \cite{KLR2013,Raphael2011-Jams}. For the reader's convenience,  we  provide the details of the adaption to our case.
		
\textbf{Step 1.} Algebraic derivation of the energy part. We compute from \eqref{equ-app2-hf-in}:
\begin{align}\label{energy-part}
&\frac{d}{dt}\left\{\frac{1}{2}\int|D^{\frac{1}{2}}\tilde{\epsilon}|^2+\frac{1}{2}\int\frac{|\tilde{\epsilon}|^2}{\lambda}-\int k(x)[F(u)-F(\tilde{Q})-F'(\tilde{Q})\cdot\tilde{\epsilon}]\right\}\notag\\
=&\Re\left(\partial_t\tilde{\epsilon},\overline{D\tilde{\epsilon}+\frac{1}{\lambda}\tilde{\epsilon}-k(x)(f(u)-f(\tilde{Q}))}\right)-\frac{\lambda_t}{2\lambda^2}\int|\tilde{\epsilon}|^2\notag\\
&-\Re\left(\partial_t\tilde{Q},k(x)\overline{(f(\tilde{Q}+\tilde{\epsilon})-f(\tilde{\epsilon})-f^{\prime}(\tilde{Q})\cdot\tilde{\epsilon})}\right)\notag\\
=&-\Im\left(\psi,\overline{D\tilde{\epsilon}+\frac{1}{\lambda}\tilde{\epsilon}-k(x)(f(u)-f(\tilde{Q}))}\right)-\frac{1}{\lambda}\Im\left(k(x)(f(u)-f(\tilde{Q})),\bar{\tilde{\epsilon}}\right)\notag\\
&-\frac{\lambda_t}{2\lambda^2}\int|\tilde{\epsilon}|^2-\Re\left(\partial_t\tilde{Q},k(x)\overline{(f(\tilde{Q}+\tilde{\epsilon})-f(\tilde{\epsilon})-f^{\prime}(\tilde{Q})\cdot\tilde{\epsilon})}\right)\notag\\
=&-\Im\left(\psi,\overline{D\tilde{\epsilon}+\frac{1}{\lambda}\tilde{\epsilon}-k(x)(2|\tilde{Q}|^2\tilde{\epsilon}-\tilde{Q}^2\bar{\tilde{\epsilon}})}\right)-\frac{1}{\lambda}\Im\int k(x)\tilde{Q}^2\bar{\tilde{\epsilon}}^2-\frac{\lambda_t}{2\lambda^2}\int|\tilde{\epsilon}|^2\notag\\
&-\Re\left(\partial_t\tilde{Q},k(x)\overline{(\bar{\tilde{Q}}\tilde{\epsilon}^2+2\tilde{Q}|\tilde{\epsilon}|^2)}\right)-\Re\left(\partial_t\tilde{Q},k(x)|\tilde{\epsilon}|^2\tilde{\epsilon}\right)\notag\\
&-\Im\left(\psi-\frac{1}{\lambda}\tilde{\epsilon},k(x)\overline{(f(\tilde{Q}+\tilde{\epsilon})-f(\tilde{Q})-f^{\prime}(\tilde{Q})\cdot\tilde{\epsilon})}\right),
\end{align}
where $f^{\prime}(\tilde{Q})\cdot\tilde{\epsilon}=2|\tilde{Q}|^2\tilde{\epsilon}+\tilde{Q}^2\bar{\tilde{\epsilon}}$.
Using \eqref{energy-priori-3-in}, we have
\begin{align}\label{energy-part-1}
-\frac{\lambda_t}{2\lambda^2}\int|\tilde{\epsilon}|^2=\frac{b}{2\lambda}\int\frac{|\tilde{\epsilon}|^2}{\lambda}+\mathcal{O}\left(\|\tilde{\epsilon}\|_{H^{1/2}}^2\right).
\end{align}
Next, from \eqref{energy-priori-1-in}, \eqref{energy-priori-2-in} and $k(x)$ is bounded, we deduce
\begin{align}\label{energy-part-2}
&\left|\Im\left(\psi-\frac{1}{\lambda}\tilde{\epsilon},k(x)\overline{(f(\tilde{Q}+\tilde{\epsilon})-f(\tilde{Q})-f^{\prime}(\tilde{Q})\cdot\tilde{\epsilon})}\right)\right|\notag\\
\lesssim&\left|\Im\left(\psi-\frac{1}{\lambda}\tilde{\epsilon},k(x)\overline{(\tilde{\epsilon}^2\bar{\tilde{Q}}+2|\tilde{\epsilon}|^2\tilde{Q}+|\tilde{\epsilon}|^2\tilde{\epsilon})}\right)\right|\notag\\
\lesssim&\left(\|\psi\|_{L^2}+\frac{1}{\lambda}\|\tilde{\epsilon}\|_{L^2}\right)\|\tilde{\epsilon}\|_{L^6}^2(\|\tilde{Q}\|_{L^6}+\|\tilde{\epsilon}\|_{L^6})\notag\\
\lesssim&\lambda\|\psi\|_{L^2}+\frac{1}{\lambda}\|\tilde{\epsilon}\|_{L^2}^2+\|\tilde{\epsilon}\|_{H^{1/2}}^2.
\end{align}
Here we used the priori bounds \eqref{energy-priori-1-in} and \eqref{energy-priori-2-in}, the inhomogeneous factor $k(x)$ is bounded and the interpolation inequality $\|f\|_{L^6}\lesssim\|f\|_{\dot{H}^{1/2}}^{\frac{2}{3}}\|f\|_{L^2}^{\frac{1}{3}}$ in $\mathbb{R}$.
		
For the cubic term hitting $\partial_t\tilde{Q}$, we replace $\partial_t\tilde{Q}$ using \eqref{equ-app2-hf-in}, integrate by parts and then rely on \eqref{energy-priori-1-in} to estimate
\begin{align}\label{energy-part-3}
\left|\int k(x)\partial_t|\tilde{\epsilon}|^2\bar{\tilde{\epsilon}}\right|\lesssim&\|\tilde{Q}\|_{\dot{H}^{3/4}}\||\tilde{\epsilon}|^2\tilde{\epsilon}\|_{\dot{H}^{1/4}}+\|\tilde{\epsilon}\|_{L^6}^3\|\tilde{\epsilon}\|_{L^6}^3+\|\psi\|_{L^2}\|\tilde{\epsilon}\|_{L^6}^3\notag\\
\lesssim&\frac{1}{\lambda^{3/4}}\|\tilde{\epsilon}\|_{L^2}^{\frac{1}{2}}\|\tilde{\epsilon}\|_{\dot{H}^{1/2}}^{\frac{5}{2}}+\frac{1}{\lambda}\|\tilde{\epsilon}\|_{L^2}\|\tilde{\epsilon}\|_{\dot{H}^{1/2}}^2+\|\psi\|_{L^2}\|\tilde{\epsilon}\|_{L^2}\|\tilde{\epsilon}\|_{\dot{H}^{1/2}}^2\notag\\
\lesssim&\|\tilde{\epsilon}\|_{H^{1/2}}^2+\lambda\|\psi\|_{L^2}.
\end{align}
Here we used the Sobolev inequality, interpolation inequality and the fractional chain rule $\|D^sF(u)\|_{L^p}\lesssim\|F^{\prime}(u)\|_{L^{p_1}}\|D^su\|_{L^{p_2}}$ for any $F\in C^1$ with $0<s\leq1$ and $1<p,p_1,p_1<\infty$ such that $\frac{1}{p}=\frac{1}{p_1}+\frac{1}{p_2}$.
		
We now insert \eqref{energy-part-1}, \eqref{energy-part-2} and \eqref{energy-part-3} into \eqref{energy-part}, we have
\begin{align}
&\frac{d}{dt}\left\{\frac{1}{2}\int|D^{\frac{1}{2}}\tilde{\epsilon}|^2+\frac{1}{2}\int\frac{|\tilde{\epsilon}|^2}{\lambda}-\int k(x)[F(u)-F(\tilde{Q})-F'(\tilde{Q})\cdot\tilde{\epsilon}]\right\}\notag\\
=&\frac{1}{\lambda}\Im\left(\int k(x)\tilde{Q}^2\bar{\tilde{\epsilon}}^2\right)-\Re\left(\partial_t\tilde{Q}k(x)\overline{(2|\tilde{\epsilon}|^2\tilde{Q}+\tilde{\epsilon}^2\bar{\tilde{Q}})}\right)+\frac{b}{2\lambda}\int\frac{|\tilde{\epsilon}|^2}{\lambda}\notag\\
&+\Im\left(\int\left(-D\psi-\frac{\psi}{\lambda}+k(x)(2|\tilde{Q}|^2\psi-\tilde{Q}^2\bar{\psi})\right)\bar{\tilde{\epsilon}}\right)\notag\\
&+\mathcal{O}\left(\lambda\|\psi\|_{L^2}^2+\lambda^{-1}\|\epsilon\|_{L^2}^2+\|\tilde{\epsilon}\|_{H^{1/2}}^2\right).
\end{align}
		
\textbf{Step 2.} Algebraic derivation of the localized Virial part .Let
\begin{align*}
\nabla\tilde{\phi}(t,x)=bA\nabla\phi\left(\frac{x}{A\lambda}\right).
\end{align*}
Then
\begin{align}\label{virial-part}
&\frac{1}{2}\frac{d}{dt}\left(b\Im\int A\nabla\phi\left(\frac{x}{\lambda}\right)\nabla\tilde{\epsilon}\bar{\tilde{\epsilon}}\right)\notag\\
=&\frac{1}{2}\Im\left(\int\partial_t\nabla\tilde{\phi}\cdot\nabla\tilde{\epsilon}\bar{\tilde{\epsilon}}\right)+\frac{1}{2}\Im\left(\int \nabla\tilde{\phi}\cdot\left((\nabla\partial_t\tilde{\epsilon})\bar{\tilde{\epsilon}}+\nabla\tilde{\epsilon}\partial_t\bar{\tilde{\epsilon}}\right)\right).
\end{align}
Using the bounds \eqref{energy-priori-3-in}, we estimate
\begin{align}\label{virial-part-1}
\left|\partial_t\nabla\tilde{\phi}\right|\lesssim|b_t|+b\left|\frac{\lambda_t}{\lambda}\right|\lesssim1.
\end{align}
Hence, by using \cite[lemma F.1]{KLR2013}, we deduce that
\begin{align}\label{virial-part-2}
\left|\frac{1}{2}\Im\left(\int\partial_t\nabla\tilde{\phi}\cdot\nabla\tilde{\epsilon}\bar{\tilde{\epsilon}}\right)\right|\lesssim\|\tilde{\epsilon}\|_{\dot{H}^{1/2}}^2+\frac{\|\tilde{\epsilon}\|_{L^2}^2}{\lambda}.
\end{align}
Now we turn to the second terms in \eqref{virial-part} containing the time derivative of $\tilde{\epsilon}$. Using \eqref{equ-app2-hf-in} and $D=D*$ is self-adjoint, a calculation yileds that
\begin{align}
	&\frac{1}{2}\Im\left(\int \nabla\tilde{\phi}\cdot\left((\nabla\partial_t\tilde{\epsilon})\bar{\tilde{\epsilon}}+\nabla\tilde{\epsilon}\partial_t\bar{\tilde{\epsilon}}\right)\right)\notag\\
	=&-\frac{1}{4}\Re\left(\int \bar{\tilde{\epsilon}}\left[-i|-i\nabla|,\nabla\tilde{\phi}\cdot(-i\nabla)+(-i\nabla)\cdot\nabla\tilde{\phi}\right]\tilde{\epsilon}\right)\notag\\
	&-b\Re\left(\int k(x)(|u|^2u-|\tilde{Q}|^2\tilde{Q})\nabla\phi\left(\frac{x}{A\lambda}\right)\cdot\overline{\nabla\tilde{\epsilon}}\right)\notag\\
	&-\frac{1}{2}\frac{b}{\lambda}\Re\left(\int k(x)(|u|^2u-|\tilde{Q}|^2\tilde{Q})\Delta\left(\frac{x}{A\lambda}\right)|\tilde{\epsilon}|^2\right)\notag\\
	&-b\Re\left(\psi\nabla\phi\left(\frac{x}{A\lambda}\right)\cdot\nabla\overline{\tilde{\epsilon}}\right)-\frac{1}{2}\frac{b}{\lambda}\Re\left(\int\psi\Delta\left(\frac{x}{A\lambda}\right)\bar{\tilde{\epsilon}}\right).
\end{align}
By the similar argument as \cite{KLR2013}, we can obtain
\begin{align*}
	&\frac{1}{2}\Im\left(\int \nabla\tilde{\phi}\cdot\left((\nabla\partial_t\tilde{\epsilon})\bar{\tilde{\epsilon}}+\nabla\tilde{\epsilon}\partial_t\bar{\tilde{\epsilon}}\right)\right)\notag\\
	=&\frac{b}{2\lambda}\int_{0}^{+\infty}\sqrt{s}\int_{\mathbb{R}}\Delta\left(\frac{x}{A\lambda}\right)|\nabla\tilde{\epsilon}_s|^2dxds-\frac{1}{8}\frac{b}{A^2\lambda^3}\int_{0}^{+\infty}\int_{\mathbb{R}}\Delta^2\Phi\left(\frac{x}{A\lambda}\right)|\tilde{\epsilon}_s|^2dxds\notag\\
	&+b\Re\left(\int A\nabla\phi\left(\frac{x}{A\lambda}\right)k(x)(2|\tilde{\epsilon}_s|^2\tilde{Q}+\tilde{\epsilon}^2\bar{\tilde{Q}})\cdot\overline{\nabla \tilde{Q}}\right)\notag\\
	&+\Im\left(\left[ibA\nabla\phi\left(\frac{x}{A\lambda}\right)\cdot\nabla\psi+i\frac{b}{\lambda}\Delta\phi\left(\frac{x}{A\lambda}\right)\psi\right]\bar{\tilde{\epsilon}}\right)\notag\\
	&+\mathcal{O}\left(\log^{\frac{1}{2}}\left(2+\|\epsilon\|_{H^{1/2}}^2\right)\|\tilde{\epsilon}\|_{H^{1/2}}^2\right),
\end{align*}
where $\epsilon_s$ satisfies equation $-\Delta \tilde{\epsilon}_s+s\tilde{\epsilon}=\sqrt{\frac{2}{\pi}}\tilde{\epsilon}$ with $s>0$. This completes the proof of this lemma.
\end{proof}

\section{Backwards Propagation of Smallness}
We now apply the energy estimate of the previous section in order to establish a bootstrap argument that will be needed in the construction of ground state mass blowup solution.
	
Let $u=u(t,x)$ be a solution to \eqref{equ-1-hf-in} defined in $[t_0,0)$. Assume that $t_0<t_1<0$ and suppose that $u$ admits on $[t_0,t_1]$ a geometrical decomposition of the form
\begin{align}\label{back-decomposition}
u(t,x)=\frac{1}{\lambda^{\frac{1}{2}}(t)}[Q_{\mathcal{P}(t)}+\epsilon]
\left(t,\frac{x}{\lambda(t)}\right)e^{i\gamma(t)},
\end{align}
where $\epsilon=\epsilon_1+i\epsilon_2$ satisfies the orthogonality condition \eqref{mod-orthogonality-condition} and $|\mathcal{P}|+\|\epsilon\|_{H^{1/2}}^2\ll1$ holds. We set
\begin{align}\label{back-small-part}
\tilde{\epsilon}(t,x)=\frac{1}{\lambda^{\frac{1}{2}}(t)}\epsilon\left(s,\frac{x}{\lambda(t)}\right)e^{i\gamma(t)}.
\end{align}
Suppose that the energy satisfies $E_0=E(u_0)>0$ and define the constant
\begin{align}\label{back-define-1}
	A_0=\sqrt{\frac{e_1}{E_0}},
\end{align}
with the constant $e_1=\frac{1}{2}(L_{-}S_{1,0},S_{1,0})>0$.
	
Now we claim that the following backwards propagation estimate holds.
\begin{lemma}(\textbf{Backwards propagation of smallness})\label{lemma-back}
Assume that, for some $t_1<0$ sufficiently close to $0$, and some $\delta\in(0,\frac{1}{4})$ fixed. We have the bounds
\begin{align*}
&\left|\|u\|_2^2-\|Q\|_2^2\right|\lesssim\lambda^2(t_1),\\
&\|D^{\frac{1}{2}}\tilde{\epsilon}(t_1)\|_2^2+\frac{\|\tilde{\epsilon}\|_2^2}{\lambda(t_1)}\lesssim\lambda(t_1),\,\,\|D^{\frac{1}{2}+\delta}\tilde{\epsilon}(t_1)\|_{L^2}^2\lesssim\lambda^{\frac{1}{2}-2\delta}(t_1),\\
&\left|\lambda(t_1)-\frac{t_1^2}{4A_0^2}\right|\lesssim\lambda^{3}(t_1),\,\,\left|\frac{b(t_1)}{\lambda^{\frac{1}{2}}(t_1)}-\frac{1}{A_0}\right|\lesssim\lambda(t_1),
\end{align*}
where $A_0$  defined in \eqref{back-define-1}. Then there exists a time $t_0<t_1$ depending on $A_0$ such that $\forall t\in[t_0,t_1]$, it holds that
\begin{align*}
&\|D^{\frac{1}{2}}\tilde{\epsilon}(t)\|_2^2+\frac{\|\tilde{\epsilon}\|_2^2}{\lambda(t)}\lesssim\|D^{\frac{1}{2}}\tilde{\epsilon}(t_1)\|_2^2+\frac{\|\tilde{\epsilon}\|_2^2}{\lambda(t_1)}\lesssim\lambda^2(t),\\
&\|D^{\frac{1}{2}+\delta}\tilde{\epsilon}(t)\|_{L^2}^2\lesssim\lambda^{\frac{1}{2}-2\delta}(t),\\
&\left|\lambda(t)-\frac{t^2}{4A_0^2}\right|\lesssim\lambda^{3}(t),\,\,\left|\frac{b(t)}{\lambda^{\frac{1}{2}}(t)}-\frac{1}{A_0}\right|\lesssim\lambda(t).
\end{align*}
\end{lemma}
\begin{proof}
By assumption, we have $u\in C^0([t_0,t_1];H^{1/2+\delta}(\mathbb{R}))$. Hence, by this continuity, let us consider a backwards time $t_0$ such that for any $t\in[t_0,t_1]$, we have bounds
\begin{align}\label{back-1}
&\|\tilde{\epsilon}\|_{L^2}\leq K\lambda(t),\,\,\|\tilde{\epsilon}(t)\|_{H^{1/2}}\leq K\lambda^{\frac{1}{2}}(t),\\\label{back-2}
&\|\tilde{\epsilon}(t)\|_{H^{\frac{1}{2}+\delta}}\leq K\lambda^{\frac{1}{4}-\delta}(t),\\\label{back-3}
&\left|\lambda(t)-\frac{t^2}{4A_0^2}\right|\leq K \lambda^{3}(t),\,\,\left|\frac{b(t)}{\lambda^{\frac{1}{2}}(t)}-\frac{1}{A_0}\right|\leq K\lambda(t),
\end{align}
for some large enough universal constant $K>0$.
		
\textbf{Step 1.} Bounds on energy and $L^2$-norm. We set
\begin{align}\label{back-1-Q}
\tilde{Q}(t,x)=\frac{1}{\lambda^{\frac{1}{2}}(t)}Q_{\mathcal{P}}\left(\frac{x}{\lambda(t)}\right)e^{i\gamma(t)}.
\end{align}
Let $J_A$ be given by \eqref{refine-energy-def}. We claim that \eqref{refine-energy-estimate} implies the following coercivity property:
\begin{align}\label{back-1-energy}
\frac{dJ_A}{dt}\geq  \frac{b}{\lambda^2}\int|\tilde{\epsilon}|^2+\mathcal{O}\left(\log^{\frac{1}{2}}\left(2+\|\tilde{\epsilon}\|_{H^{1/2}}^{-1}\right)\|\tilde{\epsilon}\|_{H^{1/2}}^2+K^4\lambda^{\frac{5}{2}}\right).
\end{align}
By the similar argument as \cite{KLR2013,Raphael2011-Jams,Raphael-2014-Duke}, we can easily obtain the following estimates:
		
Upper bound:
\begin{align}\label{back-energy-upper}
|J_A|\lesssim \|D^{\frac{1}{2}}\tilde{\epsilon}\|_{L^2}^2+\frac{1}{\lambda}\|\tilde{\epsilon}\|_{L^2}^2.
\end{align}

Lower bound:
\begin{align}\label{back-energy-lower}
J_A\geq\frac{c_0}{\lambda}\left[\|\epsilon\|_{H^{1/2}}^2-(\epsilon_1,Q)^2\right].
\end{align}
		
On the other hand, using the conservation of the $L^2$-mass and \eqref{mod-mass}, we deduce that
\begin{align}\notag
|\Re(\epsilon,Q_{\mathcal{P}})|\lesssim \|\epsilon\|_{L^2}^2+\lambda^2(t)+\left|\int |u|^2-\int Q^2\right|\lesssim\|\epsilon\|_{L^2}^2+K^2\lambda^2(t).
\end{align}
This implies
\begin{align}\label{back-1-1}
(\epsilon_1,Q)\lesssim o(\|\epsilon\|_{L^2}^2)+K^4\lambda^4(t).
\end{align}
Next, we define
\begin{align}\label{back-1-define}
X(t):=\|D^{\frac{1}{2}}\tilde{\epsilon}(t)\|_{L^2}^2+\frac{\|\tilde{\epsilon}(t)\|_{L^2}^2}{\lambda(t)}.
\end{align}
By integrating \eqref{back-1-energy} in time and using \eqref{back-energy-upper}, \eqref{back-energy-lower}, \eqref{back-1-1} and \eqref{back-1-define}, we find
\begin{align*}
X(t)\lesssim& X(t_1)+K^4\lambda^3(t)+\int_t^{t_1}\left(\log^{\frac{1}{2}}(2+\|\tilde{\epsilon}\|_{H^{1/2}}^{-1})\|\tilde{\epsilon}(\tau)\|+K^4\lambda^{\frac{5}{2}}(\tau)\right)d\tau\\
\lesssim&X(t_1)+K^4\lambda^3(t)+\int_t^{_1}\log^{\frac{1}{2}}(2+X(\tau)^{-\frac{1}{2}})X(\tau)d\tau,
\end{align*}
for $t\in[t_0,t_1]$ with some $t_0=t_0(A_0)<t_1$ close enough to $t_1<0$. By Gronwall's inequality, we deduce the desired bound for $X(t)$. In particular, we obtain
\begin{align}\label{back-1-2}
X(t)\lesssim X(t_1)+\lambda^3(t),\,\,t\in[t_0,t_1].
\end{align}
		
\textbf{Step 2.} Integration of the law for the parameters. We now integrate the law for the parameters. Indeed, lemma \ref{lemma-mod-1}, \eqref{back-3} and \eqref{back-1-2}, implies that
\begin{align}\label{back-2-1}
\left|b_s+\frac{1}{2}b^2\right|+\left|\frac{\lambda_s}{\lambda}+b\right|\lesssim\lambda^2+K^2\lambda^{\frac{5}{2}}\lesssim\lambda^2.
\end{align}
From the above bound \eqref{back-2-1}, we have
\begin{align}\label{back-2-2}
\left(\frac{b}{\lambda^{\frac{1}{2}}}\right)_s=\frac{b_s+\frac{1}{2}b^2}{\lambda^{\frac{1}{2}}}-\frac{b}{2\lambda^{\frac{1}{2}}}\left(\frac{\lambda_s}{\lambda}+b\right)\lesssim\lambda^{\frac{3}{2}}.
\end{align}
Hence, for $s<s_1$, we have
\begin{align}\label{back-2-3}
\frac{1}{A_0}-\frac{b}{\lambda^{\frac{1}{2}}}(s)\lesssim\frac{1}{A_0}-\frac{b}{\lambda^{\frac{1}{2}}}+\lambda^2(s)\lesssim\lambda(s),
\end{align}
where we used the assumption at time $t=t_1$. We now write the conservation of energy at $t$ and add the conservation of the mass, we deduce
\begin{align*}
b^2e_12=\lambda E_0+\frac{\lambda^2}{8}\int k^{\prime\prime}(0)y^2Q^4+k(0)\left(\int|u|^2-\int Q^2\right)+\mathcal{O}(\lambda^2),
\end{align*}
Since from the choice of $A_0$ in \eqref{back-define-1} and $|\|u\|_{L^2}^2-\|Q\|_{L^2}^2|=\mathcal{O}(\lambda^2)$, we have
\begin{align}\notag
\frac{1}{A_0^2}-\frac{b^2}{\lambda}+\lambda\lesssim\frac{1}{A_0}-\frac{b}{\lambda^{\frac{1}{2}}}+\lambda\lesssim\lambda,
\end{align}
where we used \eqref{back-2-3} in the last step. This together with \eqref{back-2-3} again yields
\begin{align}\label{back-2-4}
\left|\frac{b}{\lambda^{\frac{1}{2}}}-\frac{1}{A_0}\right|\lesssim\lambda.
\end{align}
From the \eqref{back-2-1} for the scaling parameter, we have
\begin{align*}
-\lambda_t=b+\mathcal{O}(\lambda^2)=\frac{\lambda^{\frac{1}{2}}}{A_0}+\mathcal{O}(\lambda^{\frac{3}{2}}+\lambda^2)=\frac{\lambda^{\frac{1}{2}}}{A_0}+\mathcal{O}(\lambda^{\frac{3}{2}})=\frac{\lambda^{\frac{1}{2}}}{A_0}+\mathcal{O}(t^3).
\end{align*}
Hence, using \eqref{back-2-4}, we have
\begin{align*}
\left|\lambda^{\frac{1}{2}}(t)-\frac{t}{2A_0}\right|\lesssim\left|\lambda^{\frac{1}{2}}(t_1)-\frac{t_1}{2A_0}\right|+\mathcal{O}(t^3)\lesssim t^2.
\end{align*}
Therefore, we obtain the desired bound for $\lambda$.
		
\textbf{Step 3.} Coercivity of the quadratic form in the RHS of \eqref{refine-energy-estimate}. We now turn to the proof of \eqref{back-1-energy}. Let
$\mathcal{H}_A(\tilde{\epsilon})$ denote the quadratic terms in $\tilde{\epsilon}$ on the right-hand side in \eqref{refine-energy-estimate}, that is
\begin{align}\label{back-3-1}
\mathcal{H}_A(\tilde{\epsilon}):=&-\frac{1}{\lambda}\Im\left(\int k(x)\tilde{Q}^2\bar{\tilde{\epsilon}}^2 \right)-\Re\left(\int k(x)\tilde{Q}_t\overline{(2|\tilde{\epsilon}|^2\tilde{Q}+\tilde{\epsilon}^2\bar{\tilde{Q}}})\right)\notag\\
&+\frac{b}{2\lambda}\int\frac{|\tilde{\epsilon}|^2}{\lambda}+\frac{b}{2\lambda}\int_{s=0}^{+\infty}\sqrt{s}\int \Delta\phi\left(\frac{x}{A\lambda}\right)|\nabla \tilde{\epsilon}_s|^2dxds\notag\\
&-\frac{1}{8}\frac{b}{A^2\lambda^3}\int_{s=0}^{\infty}\sqrt{s}\int \Delta^2\phi\left(\frac{x}{A\lambda}\right)|\tilde{\epsilon}_s|^2dxds\notag\\
&+b\Re\left(\int A\nabla\phi\left(\frac{x}{A\lambda}\right)k(x)(2|\tilde{\epsilon}|^2\tilde{Q}+\tilde{\epsilon}^2\bar{\tilde{Q}})\cdot\overline{\nabla \tilde{Q}}\right).
\end{align}
We now claim the following estimate holds:
\begin{align}\label{back-3-claim}
\mathcal{H}_A(\tilde{\epsilon})\geq\frac{c}{\lambda^{\frac{3}{2}}}\int|\epsilon|^2+\mathcal{O}(K^4\lambda^{\frac{5}{2}}),
\end{align}
with some constant $c>0$ and $\epsilon$ is given by \eqref{back-small-part}.
		
Indeed, first observe from lemma \ref{lemma-mod-1} we obtain
\begin{align}\label{back-3-mod}
	|Mod(t)|\lesssim K^2\lambda^2(t).
\end{align}
Using this estimate \eqref{back-3-mod}, we then compute from \eqref{back-1-Q}:
\begin{align*}
\partial_t\tilde{Q}=&e^{i\gamma(t)}\frac{1}{\lambda^{\frac{1}{2}}}\left[-\frac{\lambda_t}{\lambda}\Lambda Q_{\mathcal{P}}+i\gamma_tQ_{\mathcal{P}}+b_t\frac{\partial Q_{\mathcal{P}}}{\partial b}\right]\left(\frac{x}{\lambda}\right)-\frac{1}{\lambda^{\frac{1}{2}}}e^{i\gamma}Q_{\mathcal{P}}\left(\frac{x}{\lambda}\right)\\
=&\left(\frac{i}{\lambda}+\frac{b}{2\lambda}\right)\tilde{Q}+b\left(\frac{x}{\lambda}\right)\cdot\nabla\tilde{Q}+\mathcal{O}(K\lambda^{-\frac{1}{2}}).
\end{align*}
We now estimate the term
\begin{align*}
&-\Re\left(\int k(x)\tilde{Q}_t\overline{(2|\tilde{\epsilon}|^2\tilde{Q}+\tilde{\epsilon}^2\bar{\tilde{Q}}})\right)\notag\\
=&\frac{1}{\lambda}\Im\left(\int k(x)\tilde{Q}\overline{(2|\tilde{\epsilon}|^2\tilde{Q}+\tilde{\epsilon}^2\bar{\tilde{Q}})}\right)-\frac{b}{2\lambda}\Re\left(\int k(x){(2|\tilde{\epsilon}|^2\tilde{Q}+\tilde{\epsilon}^2\bar{\tilde{Q}})}\bar{\tilde{Q}}\right)\\
&-b\Re\left(\int\left(\frac{x}{\lambda}\right) k(x){(2|\tilde{\epsilon}|^2\tilde{Q}+\tilde{\epsilon}^2\bar{\tilde{Q}})}\cdot\overline{\nabla\tilde{Q}}\right)+\mathcal{O}(K\lambda^{-1}\|\epsilon\|_{L^2}^2).
\end{align*}
Plugging this estimate into \eqref{back-3-1}, we can obtain
\begin{align*}
\mathcal{H}_A(\tilde{\epsilon})=&\frac{b}{2\lambda^2}\Bigg[\int_{s=0}^{+\infty}\sqrt{s}\int\Delta\phi\left(\frac{x}{A}\right)|\nabla\epsilon_s|^2dxds+\int|\epsilon|^2\\
&-\int k(x)\left((|Q_{\mathcal{P}}|^2+2Q_{1\mathcal{P}}^2)\epsilon_1^2+4Q_{1\mathcal{P}}Q_{2\mathcal{P}}\epsilon_1\epsilon_2+(|Q_{\mathcal{P}}|^2+2Q_{2\mathcal{P}}^2)\epsilon_2^2\right)\\
&-\frac{1}{4A^2}\int_{s=0}^{\infty}\sqrt{s}\int\Delta^2\phi\left(\frac{x}{A}\right)|\epsilon_s|^2dxds\\
&+2\Re\left(\int k(x)\left(A\nabla\phi\left(\frac{x}{A}\right)-x\right)(2|\epsilon|^2Q_{\mathcal{P}}+\epsilon^2\bar{Q}_{\mathcal{P}})\cdot\overline{\nabla Q_{\mathcal{P}}}\right)\Bigg]\\
&+\mathcal{O}(K\lambda^{-1}\|\epsilon\|_{L^2}^2).
\end{align*}
Note that $\left(A\nabla\phi\left(\frac{x}{A}\right)-x\right)\equiv0$ for $|x|\leq A$ and we estimate
\begin{align*}
    &\left|\int k(x)\left(A\nabla\phi\left(\frac{x}{A}\right)-x\right)(2|\epsilon|^2Q_{\mathcal{P}}+\epsilon^2\bar{Q}_{\mathcal{P}})\cdot\overline{\nabla Q_{\mathcal{P}}}\right|\\
    \lesssim&\|(A+|x|)Q_{\mathcal{P}}\|_{L^\infty(|x|\geq A)}\|\nabla Q_{\mathcal{P}}\|_{L^\infty}\|\epsilon\|_{L^2}^2\\
    \lesssim&\left\|\frac{A+|x|}{1+|x|^2}\right\|_{L^\infty(|x|\geq A)}\|\epsilon\|_{L^2}^2\lesssim\frac{1}{A}\|\epsilon\|_{L^2}^2,
\end{align*}
where we use the uniform decay estimate $|Q_{\mathcal{P}}(x)|\lesssim\langle x\rangle^{-2}$ and the bound $\|\nabla Q_{\mathcal{P}}\|_{L^\infty}\lesssim\|Q_{\mathcal{P}}\|_{H^2}\lesssim1$.
Using the lemma \ref{lemma-app-c-2}, lemma \ref{lemma-app-c-3} and  the definitions of $L_{+,A}$ and $L_{-,A}$, see \eqref{app-c-define-1} and \eqref{app-c-define-2}, respectively, we deduce that
\begin{align*}
\mathcal{H}_A(\tilde{\epsilon})&=\frac{b}{2\lambda^2}\Bigg[(L_{+,A}\epsilon_1,\epsilon_1)+(L_{-,A}\epsilon_2,\epsilon_2)+\mathcal{O}\left(\frac{1}{A}\int|\epsilon|^2\right)\Bigg]+\frac{1}{\lambda^{\frac{3}{2}}}\mathcal{O}(K\lambda^{\frac{1}{2}}\|\epsilon\|_{L^2}^2)\\
&\gtrsim \frac{1}{\lambda^{\frac{3}{2}}}\left(\int|\epsilon|^2-(\epsilon_1,Q)^2\right)\\
&\gtrsim\frac{1}{\lambda^{\frac{3}{2}}}\int|\epsilon|^2+\mathcal{O}(K\lambda^{\frac{5}{2}}),
\end{align*}
where  we used $b\sim\lambda^{\frac{1}{2}}$ and \eqref{back-1-1}. This completes the claim \eqref{back-3-claim}.
		
\textbf{Step 4.} Controlling the remainder terms in \eqref{refine-energy-estimate}. It remains to control the $\psi$ terms in \eqref{refine-energy-estimate}. According to \eqref{equ-refine-approximate-1}, \eqref{back-1-Q} and the construction of $Q_{\mathcal{P}}$, we have
\begin{align*}
\psi=&\frac{1}{\lambda^{\frac{3}{2}}}\Bigg[i\left(b_s+\frac{b^2}{2}\right)\partial_bQ_{\mathcal{P}}+i\left(\frac{\lambda_s}{\lambda}+b\right)\lambda\partial_{\lambda}Q_{\mathcal{P}}\\
&+i\left(\frac{\lambda_s}{\lambda}+b\right)\Lambda Q_{\mathcal{P}}+\tilde{\gamma}_sQ_{\mathcal{P}}+\Psi_{\mathcal{P}}\Bigg]\left(\frac{x}{\lambda}\right)e^{i\gamma}.
\end{align*}
By the lemma \ref{lemma-3app} and \eqref{back-3-mod}, we have the rough bound on $\psi$:
\begin{align}
\|\nabla^j\psi\|_{L^2}\lesssim K^2\lambda^{1-j},\,\,\text{for}\,\,j=0,1
\end{align}
Write $\psi=\psi_1+\psi_2$. Then using the similar argument as \cite{KLR2013,Raphael2011-Jams}, we can obtain that
\begin{align*}
&\left|\Im\left(\int\left[-D\psi_2-\frac{\psi_2}{\lambda}+k(x)(2|\tilde{Q}|^2\psi_2-\tilde{Q}^2\bar{\psi}_2)\right]\bar{\tilde{\epsilon}}\right)\right|\\
\lesssim&K^2\lambda^{\frac{1}{2}}\|\epsilon\|_{L^2}\lesssim o\left(\frac{\|\epsilon\|_{L^2}^2}{\lambda^{\frac{3}{2}}}\right)+K^4\lambda^{\frac{5}{2}}
\end{align*}
and
\begin{align*}
&\left|\Im\left(\int\left[-D\psi_1-\frac{\psi_1}{\lambda}+k(x)(2|\tilde{Q}|^2\psi_1-\tilde{Q}^2\bar{\psi}_1)\right]\bar{\tilde{\epsilon}}\right)\right|\\
\lesssim&K^2\lambda^{\frac{1}{2}}\|\epsilon\|_{L^2}\lesssim o\left(\frac{\|\epsilon\|_{L^2}^2}{\lambda^{\frac{3}{2}}}\right)+K^4\lambda^{\frac{5}{2}}.
\end{align*}
Finally, we recall \eqref{back-3-claim} and insert all the derived estimates for the terms involving $\psi$ in \eqref{refine-energy-estimate} and we conclude that the coercivity property \eqref{back-1-energy} holds.
		
\textbf{Step 5.} Bounds on $\|D^{\frac{1}{2}+\delta}\tilde{\epsilon}(t)\|_{L^2}$. By the similar argument as \cite{KLR2013}, we can obtain the $H^{\frac{1}{2}+\delta}$ estimate. Here we omit the details.
\end{proof}
	
\section{Proof of the Theorem \ref{Theorem-1-hf-2}.}
In this section, we prove the following result.
\begin{theorem}
Let $\gamma_0$ and $E_0>0$ be given. Then there exist a time $t_0<0$ and a solution $u_c\in C^0([t_0,0); H^{\frac{1}{2}}(\mathbb{R}))$ of \eqref{equ-1-hf-in} such that $u$ blowup  at time $T=0$ with
\begin{align}\notag
E(u_c)=E_0,\,\,\text{and}\,\, \|u_c\|_2^2=\|Q\|_2^2.
\end{align}
Furthermore, we have $\|D^{\frac{1}{2}}u_c\|_2\sim t^{-1}$ as $t\rightarrow0^{-}$, and $u_c$ is of the form	\begin{align}\notag
u_c(t,x)=\frac{1}{\lambda_c^{\frac{1}{2}}(t)}[Q_{\mathcal{P}_c(t)}+\epsilon_c]
\left(t,\frac{x}{\lambda}\right)e^{i\gamma(t)}=\tilde{Q}_c+\tilde{\epsilon}_c,
\end{align}
where $\mathcal{P}_c(t)=(b_c(t),\lambda_c(t))$, and $\epsilon_c$ satisfies the orthogonality condition \eqref{mod-orthogonality-condition}. Finally, the following estimate hold:
\begin{align*}
&\|\tilde{\epsilon}_c\|_2\lesssim\lambda_c,\ \|\tilde{\epsilon}_c\|_{H^{1/2}}\lesssim\lambda_c^{\frac{1}{2}},\\
&\lambda_c(t)-\frac{t^2}{4A_0^2}=\mathcal{O}(\lambda_c^3),\ \frac{b_c}{\lambda_c^{\frac{1}{2}}}(t)-\frac{1}{A_0}=\mathcal{O}(\lambda_c),\\
&\,\,\gamma(t)=-\frac{4A_0^2}{t}+\gamma_0+\mathcal{O}(\lambda^{\frac{1}{2}}).
\end{align*}
for $t\in[t_0,0)$ and $t$ sufficiently close to $0$.
Here $A_0>0$ is the constant  defined in \eqref{back-define-1}.
\end{theorem}
\begin{proof}
By the similar argument as \cite{Martel-2005-Amer,Martel-2006-poincare,Merle-1990-CMP,Raphael2011-Jams,KLR2013,Raphael-2009-cpam}, we can obtain the desired result. Here we omit the details.
\end{proof}
	
\appendix
\section{ Coercivity estimate for the localized energy}
In this section, we give some useful lemmas. By the similar argument as \cite{KLR2013}, we can easily obtain the following lemmas, so here we omit the details.
In the following, we assume that $A>0$ is a sufficiently large constant. Let $\phi:\mathbb{R}\rightarrow\mathbb{R}$ be the smooth cutoff function introduced in \eqref{energy-cutoff-function}, Section \ref{section-refined-energy}. For $\epsilon=\epsilon_1+i\epsilon_2\in H^{1/2}(\mathbb{R})$, we consider the quadratic forms
\begin{align}\label{app-c-define-1}
L_{+,A}(\epsilon_1):&=\int_{s=0}^{\infty}\sqrt{s}\int\Delta \phi_A|\nabla(\epsilon_1)_{s}|^2dxds+\int|\epsilon_{1}|^2-2\int k(x)Q|\epsilon_1|^2\\\label{app-c-define-2}
L_{-,A}(\epsilon_2):&=\int_{s=0}^{\infty}\sqrt{s}\int\Delta \phi_A|\nabla(\epsilon_2)_{s}|^2dxds+\int|\epsilon_{2}|^2-\int k(x) Q|\epsilon_2|^2,
\end{align}
where $\Delta \phi_A=\Delta(\phi\left(\frac{x}{A}\right))$. As in lemma \ref{lemma-refine-energy-in}, we denote
\begin{align}\label{app-c-define-3}
u_s=\sqrt{\frac{2}{\pi}}\frac{1}{-\Delta+s}u,\  \ \text{for}\ s>0.
\end{align}
We start with the following simple identity.
	
For $u\in H^{1/2}(\mathbb{R})$, we have
\begin{align}\label{app-c-identity}
\int_0^{\infty}\sqrt{s}\int_{\mathbb{R}}|\nabla u_s|^2dxds=\|D^{1/2}u\|_2^2.
\end{align}
Indeed, by applying Fubini's theorem and using Fourier transform, we find that
\begin{align}\notag
\int_0^{\infty}\sqrt{s}\int_{\mathbb{R}}|\nabla u_s|^2dxds=\frac{2}{\pi}
\int_{\mathbb{R}}\int_0^{\infty}\frac{\sqrt{s}ds}{(\xi^2+s)^2}|\xi|^2|\hat{u}(\xi)|^2d\xi=\|D^{1/2}u\|_2^2.
\end{align}
In general, we have
\begin{align}\label{app-c-identity-2}
\frac{2}{\pi}\int_0^{\infty}\sqrt{s}\int_{\mathbb{R}}|(-\Delta)^{\alpha/2}u_s|^2dxds=\|D^{\alpha-\frac{1}{2}}u\|_2^2.
\end{align}
Next, we establish a technical result, which show that, when taking the limit $A\rightarrow+\infty$, the quadratic form $\int_0^{\infty}\sqrt{s}\int\Delta\phi_A|\nabla u_s|^2dxds+\|u\|_2^2$ defines a weak topology that serves as a useful substitute for weak convergence in $H^{1/2}(\mathbb{R})$. The precise statement reads as follows.
\begin{lemma}\label{app-lemma-c-1}
Let $A_n\rightarrow\infty$ and suppose that $\{u_n\}_{n=1}^{\infty}$ is a sequence in $H^{1/2}(\mathbb{R})$ such that
\begin{align}\notag
\int_0^{\infty}\sqrt{s}\int\Delta\phi_{A_{n}}|\nabla (u_n)_s|^2dxds+\|u_n\|_2^2\leq C,
\end{align}
for some constant $C>0$ independent of $n$. Then, after possibly passing to a subsequence of $\{u_n\}_{n=1}^{\infty}$, we have that
\begin{align}\notag
u_n\rightharpoonup u\ \text{weakly in}\ L^2(\mathbb{R})\ \text{and}\ u_n\rightarrow u\ \text{strongly in}\ L^2_{loc}(\mathbb{R}),
\end{align}
and $u\in H^{1/2}(\mathbb{R})$. Moreover, we have the bound
\begin{align}\notag
\|D^{1/2}u\|_2^2\leq\liminf_{n\rightarrow\infty}\int_0^{\infty}\sqrt{s}\int\Delta\phi_{A_n}|\nabla (u_n)_s|^2dxds.
\end{align}
\end{lemma}

\begin{lemma}\label{lemma-app-c-2}
Let $L_{+,A}(\epsilon_1)$ and $L_{-,A}(\epsilon_2)$ be the quadratic forms defined \eqref{app-c-define-1} and \eqref{app-c-define-2}, respectively. Then there exist  constants $C_0>0$ and $A_0>0$ such that, for all $A\geq A_0$ and all $\epsilon=\epsilon_1+i\epsilon_2\in H^{1/2}(\mathbb{R})$, we have the coercivity estimate
\begin{align}\notag
(L_{+,A}\epsilon_1,\epsilon_1)+(L_{-,A}\epsilon_2,\epsilon_2)\geq C_0\int|\epsilon|^2-\frac{1}{C_0}\left\{(\epsilon_1,Q)^2+(\epsilon_1,S_{1,0})^2+|(\epsilon_2,\rho_1)|^2\right\}.
\end{align}
Here $S_{1,0}$ is the unique functions such that $L_{-}S_{1,0}=\Lambda Q$ with $(S_{1,0},Q)=0$, , respectively, and the function $\rho_1$ is defined in \eqref{mod-definition-rho}.
\end{lemma}

\begin{lemma}\label{lemma-app-c-3}
For any $u\in L^2(\mathbb{R})$, we have the bound
\begin{align}\notag
\left|\int_{s=0}^{+\infty}\sqrt{s}\int\Delta^2\phi_{A}|u_s|^2dxds\right|\lesssim\frac{1}{A}\|u\|_2^2.
\end{align}
\end{lemma}

\noindent	
\textbf{Acknowledgements}

YL was supported by China Postdoctoral Science Foundation (No. 2021M701365).


\begin{thebibliography}{99}
\bibitem{Banica-2011-CPDE} V. Banica, R. Carles, T. Duyckaerts,  Minimal blow-up solutions to the mass-critical inhomogeneous NLS equation. Comm. Partial Differential Equations 36(3) (2011),  487-531.
\bibitem{BGV2018} J. Bellazzini, V. Georgiev, N. Visciglia, Long time dynamics for semirelativistic NLS and half wave in arbitrary dimension, Math. Ann. 371 (2018), 707-740.
		
		
\bibitem{lenzmann-2016blowup} T. Boulenger, D. Himmelsbach, E. Lenzmann, Blowup for fractional NLS, J.  Funct. Anal. 271 (2016), 2569-2603.
\bibitem{majda2001} D. Cai, A. J. Majda, D. W. McLaughlin,  E. G. Tabak, Dispersive wave turbulence in one dimension, Phys. D 152 (2001), 551-572.
\bibitem{cho2013} Y. Cho, H. Hajaiej, G. Hwang, T. Ozawa, On the Cauchy problem of fractional Sch\"{o}dinger equation with Hartree type nonlinearity, Funkcial. Ekvac. 56(2) (2013),  193-224.
\bibitem{eckhaus1983} W. Eckhaus, P. Schuur, The emergence of solitons of the Korteweg-de Vries equation from arbitrary initial conditions, Math. Methods Appl. Sci. 5(1) (1983),  97-116.
		
\bibitem{Frank-lenzmann2013} R. Frank, E. Lenzmann, Uniqueness of non-linear ground states for fractional Laplacians in $\mathbb{R}$, Acta Math. 210(2) (2013),  261-318.
\bibitem{FrankLS2016} R. L. Frank,  E. Lenzmann, L. Silvestre, Uniqueness of radial solutions for the fractional Laplacian. Comm. Pure Appl. Math. 69 (2016),  1671-1726.

\bibitem{Georgiev-Li-2022-CPDE} V. Georgiev, Y. Li, Nondispersive solutions to the mass critical half-wave equation in two dimensions, Comm. Partial Differential Equations, (2021), 47 (2022), no. 1, 39-88.

\bibitem{Georgiev-Li-2021-JFA} V. Georgiev, Y. Li, Blowup dynamics for mass critical half-wave
equation in 3D, J. Funct. Anal. 281 (2021), 109132.
\bibitem{Gerard-2010} P. G\'erard, S. Grellier, The cubic Szeg\H{o} equation. Ann. Sci. \'Ec. Norm. Sup\'er. (4) 43(5) (2010), 761-810.
\bibitem{Hidano-2019-sel}K. Hidano, C. Wang, Fractional derivatives of composite functions and the Cauchy problem for the
nonlinear half wave equation, Sel. Math. 25 (2019) 2.
\bibitem{Ionescu2014} A. D. Ionescu, F. Pusateri, Nonlinear fractional Schr\"{o}dinger equations in one dimension, J. Funct. Anal. 266(1) (2014),  139-176.

\bibitem{klein2014} C. Klein, C. Sparber,  P. Markowich, Numerical study of fractional nonlinear Schr\"{o}dinger equations, Proc. R. Soc. A 470 (2172) (2014),  20140364.
\bibitem{k-lenzmann2013} K. Kirkpatrick, E. Lenzmann, G. Staffilani, On the continuum limit for discrete NLS with long-range lattice interactions, Comm. Math. Phys. 317(3) (2013),  563-591.
\bibitem{Raphael-2009-cpam}J. Krieger, Y. Martel, P. Rapha\"{e}l, Two-soliton solutions to the three-dimensional gravitational Hartree equation, Commun. Pure Appl. Math. 62 (11) (2009), 1501-1550.
\bibitem{KLR2013} J. Krieger, E. Lenzmann, P. Rapha\"{e}l, Nondispersive solutions to the $L^2$-critical half-wave equation, Arch. Ration. Mech. Anal. 209  (2013), 61-129.

\bibitem{majda1997} A. J. Majda, D. W. McLaughlin,  E. G. Tabak, A one-dimensional model for dispersive wave turbulence, J. Nonlinear Sci. 7(1) (1997), 9-44.
\bibitem{Martel-2005-Amer} Y. Martel,  Asymptotic N-soliton-like solutions of the subcritical and critical generalized Korteweg-de Vries equations. Amer. J. Math. 127(5) (2005), 1103-1140.
\bibitem{Martel-2006-poincare} Y. Martel, F. Merle, Multi solitary waves for nonlinear Schr\"{o}dinger equations. Ann. Inst. H. Poincar\'e Anal. Non Lin\'eaire 23(6) (2006),  849-864.

\bibitem{M1996}F. Merle,  Nonexistence of minimal blow-up solutions of equations $iu_t=-\Delta u-K(x)|u|^{4/N}u $ in $R^N$. Ann. Inst. H. Poincar\'e Phys. Th\'eor. 64 (1996), no. 1, 33-85.
\bibitem{Merle-1993-Duke}F. Merle, Determination of blow-up solutions with minimal mass for nonlinear Schr\"odinger equations with critical power. Duke Math. J. 69(2) (1993), 427-454.
\bibitem{Merle-1990-CMP} F. Merle, Construction of solutions with exactly $k$ blow-up points for the Schr\"{o}dinger equation with critical nonlinearity. Commun. Math. Phys. 129(2) (1990), 223-240.
\bibitem{MerleR2006} F, Merle, P. Rapha\"{e}l,  On a sharp lower bound on the blow-up rate for the $L^2$ critical nonlinear Schr\"{o}dinger equation. J. Amer. Math. Soc. 19 (2006), 37-90.
\bibitem{Raphael-2014-Duke} F. Merle, P. Rapha\"{e}l,  J. Szeftel,  On collapsing ring blow-up solutions to the mass supercritical nonlinear Schr\"{o}dinger equation. Duke Math. J. 163(2) (2014),  369-431.
\bibitem{Rephael-2007-poincare} F. Planchon, P. Rapha\"{e}l,  Existence and stability of the log-log blow-up dynamics for the $L^2$-critical nonlinear Schr\"{o}dinger equation in a domain, Ann. Henri Poincar\'{e} 8(6) (2007),  1177-1219.

\bibitem{Raphael2011-Jams} P. Rapha\"{e}l, J. Szeftel, Existence and uniqueness of minimal blow-up solutions to an inhomogeneous mass critical NLS. J. Amer. Math. Soc. 24 (2) (2011), 471-546.
\bibitem{stein-1993} E. M. Stein, Harmonic analysis: real-variable methods, orthogonality, and oscillatory intergral. Princeton Mathematical series, vol. 43, Princeton University Press 1993.
		
\bibitem{Weinstein1985} M. I. Weinstein, Modulational stability of ground states of nonlinear Schr\"{o}dinger equations. SIAM J. Math. Anal. 16(3)  (1985), 472-491.
\bibitem{Weinstein1987} M. I. Weinstein, Existence and dynamic stability of solitary wave solutions of equations arising in long wave propagation, Comm. Partial Differential Equations 12(10) (1987),  1133-1173.
		
		

		

		

		

		
		
	\end{thebibliography}
\end{document}